\documentclass[reqno,twoside]{amsart}
\usepackage{amsfonts}
\usepackage{amsmath,amssymb}
\usepackage[dvips,bottom=1.4in,right=1in,top=1in, left=1in]{geometry}
\usepackage{epic}
\usepackage{pstricks}
\usepackage{graphicx}
\usepackage{xcolor}

\usepackage{amsfonts}
\usepackage{amsmath}
\usepackage{amsthm}
\usepackage{etoolbox}
\usepackage{color}
\usepackage[english]{babel}
\usepackage{hyperref}
\usepackage{microtype}
\usepackage{fullpage}
\usepackage{yfonts}
\usepackage{microtype}
\usepackage{subfig}
\usepackage{epstopdf}

\AppendGraphicsExtensions{.tif}

\DisableLigatures{encoding = *, family = * }

\newtheorem{thm}{Theorem}[section]
\newtheorem{cor}[thm]{Corollary}
\newtheorem{lem}[thm]{Lemma}

\newtheorem{defn}[thm]{Definition}
\newtheorem{rem}[thm]{Remark}

\numberwithin{equation}{section}

\def\NN{\mathbb{N}}
\def\ZZ{\mathbb{Z}}

\def\RR{\mathbb{R}}
\def\CC{\mathbb{C}}
\def\O{\mathcal{O}}
\def\Q{\mathcal{Q}}
\def\I{\mathcal{I}}
\def\zed{\mathcal{Z}}
\def\sg{\sigma}

\newcommand{\norm}[2]{{\left\|#1\right\|}_{#2}}
\newcommand{\ffl}[2]{(-d_x^{\,2})^{#1}#2}

\newcommand{\Hs}[1]{\mathbb{H}_s^{#1}(-1,1)}

\newcommand{\hs}[1]{\mathbb{H}_s^{#1}}
\newcommand{\ttilde}[1]{\widetilde{\widetilde{#1}}}

\newcommand{\sgn}{\textrm{sgn}}


\begin{document} 

\title{Null-controllability properties of a fractional wave equation with a memory term}

\author[rvt,rvt1]{Umberto Biccari}

\author[rvt2]{Mahamadi Warma}

\address{U. Biccari, DeustoTech, University of Deusto, 48007 Bilbao, Basque Country, Spain.}
\address{U. Biccari, Facultad de Ingenier\'ia, Universidad de Deusto, Avenida de las Universidades 24, 48007 Bilbao, Basque Country, Spain. }
\email{umberto.biccari@deusto.es -- u.biccari@gmail.com}
\address{M. Warma, University of Puerto Rico, Rio Piedras Campus, Department of Mathematics,
 Faculty of Natural Sciences,  17 University AVE. STE 1701  San Juan PR 00925-2537 (USA)}
\email{mahamadi.warma1@upr.edu -- mjwarma@gmail.com}

\begin{abstract}
We study the null-controllability properties of a one-dimensional wave equation with memory associated with the fractional Laplace operator. The goal is not only to drive the displacement and the velocity to rest at some time-instant but also to require the memory term to vanish at the same time, ensuring that the whole process reaches the equilibrium. The problem being equivalent to a coupled nonlocal PDE-ODE system, in which the ODE component has zero velocity of propagation, we are required to use a moving control strategy. Assuming that the control is acting on an open subset $\omega(t)$ which is moving with a constant velocity $c\in\mathbb{R}$, the main result of the paper states that the equation is null controllable in a sufficiently large time $T$ and for initial data belonging to suitable fractional order Sobolev spaces. The proof will use a careful analysis of the spectrum of the operator associated with the system and an application of a classical moment method.
\end{abstract}

\maketitle

\section{Introduction, well-posedness and main results}\label{into_sect}

Let $T>0$ be a real number, $Q:=(0,T)\times(-1,1)$, $Q^c:=(0,T)\times(-1,1)^c$ where $(-1,1)^c:=\RR\setminus(-1,1)$, and let $M\in\RR$. Of concern in this paper is  the analysis of the controllability properties of the following fractional wave equation involving a memory term:
\begin{align}\label{wave_mem}
	\begin{cases}
		\displaystyle y_{tt}(t,x)+\ffl{s}{y}(t,x) - M\int_0^t \ffl{s}{y}(\tau,x)\,d\tau = \mathbf{1}_{\omega(t)}u(t,x) & (t,x)\in Q,
		\\
		y(t,x)=0 & (t,x)\in Q^c,
		\\
		y(0,x)=y^0(x),\;\;y_t(0,x)=y^1(x) & x\in(-1,1).
	\end{cases}
\end{align}

In \eqref{wave_mem}, the memory enters in the principal part, and the control is applied on an open subset $\omega(t)$ of the domain $(-1,1)$ where the waves propagate. The support $\omega(t)$ of the control $u$ at time $t$ moves in space with a constant velocity $c$, that is,
\begin{align*}
	\omega(t) = \omega_0-ct,
\end{align*}
with $\omega_0\subset(-1,1)$ a reference set, open and non empty. The control $u\in L^2(\O)$ is then an applied force localized in $\omega(t)$, where
\begin{align*}
	\O:=\Big\{(t,x):\,\;\, t\in(0,T), x\in\omega(t)\Big\}.
\end{align*}

Moreover, with $\ffl{s}{}$, $s\in (0,1)$, we denote the fractional Laplace operator whose precise definition will be given in the next section.

Evolution equations involving memory terms are an effective tool for modeling a large spectrum of phenomena which apart from their current state are influenced also by their history. They appear in several different applications, including viscoelasticity, non-Fickian diffusion, and thermal processes with memory (see \cite{pruss2013evolutionary,renardy1987mathematical} and the references therein).   Fractional order operators have recently emerged as a modeling alternative in various branches of science. From the long list of phenomena which are more appropriately modeled by fractional differential equations, we mention: viscoelasticity, anomalous transport and diffusion, hereditary phenomena with long memory, nonlocal electrostatics, the latter being relevant to drug design, and L\'evy motions which appear in important models in both applied mathematics and applied probability. A number of stochastic models for explaining anomalous diffusion have been also introduced in the literature; among them we  quote the fractional Brownian motion; the continuous time random walk;  the L\'evy flights; the Schneider grey Brownian motion; and more generally, random walk models based on evolution equations of single and distributed fractional order in  space (see e.g. \cite{DS,GR,Man,Sch,ZL}). In general, a fractional diffusion operator corresponds to a diverging jump length variance in the random walk. We refer to \cite{dihitchhiker,Val} and the references therein for a complete analysis, the derivation and the applications of the fractional Laplace operator. For further details we also refer to \cite{GW-F,GW-CPDE} and their references.

Controllability problems for evolution equations with memory terms have been extensively studied in the past. Among other contributions, we mention \cite{kim1993control,leugering1984exact,leugering1987exact,loreti2012boundary, loreti2010reachability,mustafa2015control,pandolfi2013boundary,romanov2013exact} which, as in our case, deal with hyperbolic type equations. Nevertheless, in the majority of these works the issue has been addressed focusing only on the steering of the state of the system to zero at time $T$, without considering that the presence of the memory introduces additional effects that makes the classical controllability notion not suitable in this context. Indeed, driving the solution of \eqref{wave_mem} to zero is not sufficient to guarantee that the dynamics of the system reaches an equilibrium. If we were considering an equation without memory, once its solution is driven to rest at time $T$ by a control, then it vanishes for all $t\geq T$ also in the absence of control. On the other hand, the introduction of a memory term may produce accumulation effects that affect the stability of the system.
For these reasons, in some recent papers (see, e.g., \cite{biccari2018null,chaves2017controllability,lu2017null}), the classical notion of controllability for a wave equation, requiring the state and its velocity to vanish at time $T$, has been extended with the additional imposition that the control shall \textit{shut down} also the memory effects, which in our case corresponds to the condition
\begin{align*}
	\int_0^T \ffl{s}{y}(\tau,x)\,d\tau = 0.
\end{align*}

This special notion of controllability is generally called \textit{memory-type null controllability}. The aim of the present article is  to completely analyze  the controllability properties of \eqref{wave_mem} in the above mentioned framework.  Our main result states that if $\frac 12<s<1$, then the system \eqref{wave_mem} is memory-type null controllable for large enough time $T$. In another words, there exists a control function $u$ such that  for $T$ large enough, the unique solution $y$ of \eqref{wave_mem} satisfies
\begin{align*}
y(x,T)=y_t(x,T)=\int_0^T \ffl{s}{y}(\tau,x)\,d\tau = 0 \;\mbox{ for a.e. }\; x\in(-1,1).
\end{align*}


Our approach is inspired from the techniques presented in our recent work \cite{biccari2018null} for the Laplace operator, suitably adapted in order to deal with the additional nonlocal features (in addition to the memory effects), introduced into our model by the fractional Laplacian. In more detail, the technique we will use is based on a spectral analysis and an explicit construction of biorthogonal sequences. Although this approach limits our study to a one-dimensional case, it has the additional advantage of offering new insights on the behavior of this type of problems through the detailed study of the properties of the spectrum. Besides, as in other related previous works, we shall view the wave model \eqref{wave_mem} as the coupling of a wave-like PDE with an ODE. This approach will enhance the necessity of a moving control strategy. Indeed, we will show that the memory-type null controllability of the system fails if the support $\mathcal O$ of the control $u$ is time-independent, unless of course in the trivial case where $\mathcal O = Q$. We mention that this strategy of a moving control has been successfully used in the past in the framework of viscoelasticity, the structurally damped wave equation and the Benjamin-Bona-Mahony equation (see e.g. \cite{chaves2014null,martin2013null,rosier2013unique}). 

The mains challenges and novelties of the present article are the following.
\begin{itemize}
\item Here we have to deal with the fractional Laplace operator which is a nonlocal pseudo-differential operator, hence, this case is more challenging than the case of the Laplace operator which is a local differential operator.

\item As we have mentioned above, the proof of the controllability result will be based on a careful analysis of the spectrum of the operator associated to the system.  But here, there is no explicit formula for the eigenvalues and eigenfunctions of the fractional Laplace operator with the zero Dirichlet exterior condition.  This makes it more difficult than the case of the Dirichlet Laplace operator on an interval where the eigenvalues and eigenfunctions are well-known and very easy to compute.

\item This is the first work  that studies moving control problems (with or without memory terms) associated with the fractional Laplace operator.
\end{itemize}


\subsection{Functional setting}

In this section, we introduce some notations, define the function spaces in which the wave equation \eqref{wave_mem}  with memory is well-posed.
We start by giving a rigorous definition of the fractional Laplace operator. To this end, for $0<s<1$, let us consider the space
\begin{align*}
	\mathcal L^1_s(\RR) :=\left\{ u:\RR\longrightarrow\RR\,:\; u\textrm{ measurable },\;\int_{\RR}\frac{|u(x)|}{(1+|x|)^{1+2s}}\,dx<\infty\right\}.
\end{align*}
For $u\in\mathcal L_s^1(\RR)$ and $\varepsilon>0$, we set 
\begin{align*}
(-d_x^{\,2})^s_{\varepsilon}\, u(x) = C_s\,\int_{|x-y|>\varepsilon}\frac{u(x)-u(y)}{|x-y|^{1+2s}}\,dy,\;\;\; x\in\RR.
\end{align*}
The fractional Laplacian is then defined by the following singular integral
\begin{align}\label{fl}
	\ffl{s}{u}(x) = C_s\,\mbox{P.V.}\,\int_{\RR}\frac{u(x)-u(y)}{|x-y|^{1+2s}}\,dy = \lim_{\varepsilon\to 0^+} (-d_x^2)^s_{\varepsilon} u(x), \;\;\; x\in\RR,
\end{align}
provided that the limit exists. Here $C_s$ is an explicit normalization constant given by
\begin{align*}
	C_s:=\frac{s2^{2s}\Gamma\left(\frac{1+2s}{2}\right)}{\sqrt{\pi}\Gamma(1-s)},
\end{align*}
$\Gamma$ being the usual Euler Gamma function. 

We notice that if $0<s<\frac 12$ and $u$ is a smooth function, for example bounded and Lipschitz continuous on $\RR$, then the integral in \eqref{fl} is in fact not really singular near $x$ (see e.g. \cite[Remark 3.1]{dihitchhiker}). Moreover, $\mathcal L_s^1(\RR)$ is the right space for which $v:=(-d_x^{\,2})^s_{\varepsilon}\, u$ exists for every $\varepsilon > 0$, $v$ being also continuous at the continuity points of $u$.
We also mention that the fractional Laplace operator can be also defined as the pseudo-differential operator with symbol $|\xi|^{2s}$. For more details on the fractional Laplacian we refer to \cite{dihitchhiker,SV2,Val,War-PA,warma} and their references.

Let us now define the function spaces in which we are going to work. Let $\Omega\subset\mathbb R$ be an arbitrary open set and $0<s<1$. We let
\begin{align*}
	H^s(\Omega):=\left\{u\in L^2(\Omega):\;\int_{\Omega}\int_{\Omega}\frac{|u(x)-u(y)|^2}{|x-y|^{1+2s}}\;dxdy<\infty\right\}
\end{align*}
be the fractional order Sobolev space endowed with the norm defined by
\begin{align*}
	\|u\|_{H^s(\Omega)}^2=\int_{\Omega}|u|^2\;dx+\int_{\Omega}\int_{\Omega}\frac{|u(x)-u(y)|^2}{|x-y|^{1+2s}}\;dxdy.
\end{align*}
We let
\begin{align*}
	H_0^s(\overline\Omega)=\Big\{u\in H^s(\mathbb R):\; u=0\;\mbox{ in }\;\mathbb R\setminus\Omega\Big\}.
\end{align*}
Let $(-d_x^{\,2})^s_D$ be the self-adjoint operator on $L^2(\Omega)$  associated with the closed and bilinear form
\begin{align*}
	\mathcal E(u,v)=\frac{C_s}{2}\int_{\mathbb R}\int_{\mathbb R}\frac{(u(x)-u(y))(v(x)-v(y))}{|x-y|^{1+2s}}\;dxdy,\;\;u,v\in H_0^s(\overline{\Omega}).
\end{align*}
More precisely,
\begin{align*}
	D((-d_x^{\,2})^s_D)=\Big\{u\in H_0^s(\overline\Omega):\; (-d_x^{\,2})^su\in L^2(\Omega)\Big\},\; (-d_x^{\,2})^s_Du=(-d_x^{\,2})^su.
\end{align*}

Then $(-d_x^{\,2})^s_D$ is the realization in $L^2(\Omega)$ of the fractional Laplace operator $(-d_x^{\,2})^s$ with the zero Dirichlet exterior condition $u=0$ in $\Omega^c:=\RR^N\setminus\Omega$. It is well-known (see e.g. \cite{SV2}) that $(-d_x^{\,2})^s_D$ has a compact resolvent and its eigenvalues form a non-decreasing sequence of real numbers $0<\rho_1\leq\rho_2\leq\cdots\leq\rho_n\leq\cdots$ satisfying $\lim_{n\to+\infty}\rho_n=+\infty$. In addition, the eigenvalues are of finite multiplicity. Let $(e_n)_{n\geq 1}$ be the orthonormal basis of eigenfunctions associated with $(\rho_n)_{n\geq 1}$, that is, 
\begin{align}\label{fl_eigen}
	\begin{cases}
		\ffl{s}{e_n}=\rho_n e_n & \mbox{ in } (-1,1),
		\\
		e_n = 0 & \mbox{ in } (-1,1)^c.
	\end{cases}
\end{align}
For any real $\sigma\ge 0$ we define the space $\mathbb H_s^{\sigma}(\Omega)$ as the domain of the $\sigma$-power of  $(-d_x^{\,2})^s_D$ . More precisely, 
\begin{align}\label{Hsigma}
   \mathbb H_s^{\sigma}(\Omega):=\left\{u\in L^2(\Omega)\;:\; \sum_{n\geq 1}\left|\rho_n^\sigma(u,e_n)_{L^2(\Omega)}\right|^2<+\infty\right\},
\end{align}
and 
\begin{align*}
	\norm{u}{\mathbb H_s^{\sigma}(\Omega)}:=\left(\sum_{n\geq 1}\left|\rho_n^\sigma(u,e_n)_{L^2(\Omega)}\right|^2\right)^{\frac 12}.
\end{align*}

In this setting it is easy to see that $\mathbb H_s^{\frac 12}(\Omega)=H_0^s(\overline{\Omega})$ with equivalent norms.
We shall denote by $\mathbb H_s^{-\sigma}(\Omega)= \left(\mathbb H_s^{\sigma}(\Omega)\right)'$ the dual of $\mathbb H_s^{\sigma}(\Omega)$ with respect to the pivot space $L^2(\Omega)$ and we endow it with the norm 
\begin{align*}
	\norm{u}{\mathbb H_s^{-\sigma}(\Omega)}:=\left(\sum_{n\geq 1}\left|\rho_n^{-\sigma}(u,e_n)_{L^2(\Omega)}\right|^2\right)^{\frac 12}.
\end{align*}
Then we have the following continuous embeddings $\mathbb H_s^{\sigma}(\Omega)\hookrightarrow L^2(\Omega)\hookrightarrow \mathbb H_s^{-\sigma}(\Omega)$.
%

We shall also need the following asymptotic result of the eigenvalues which proof is contained in \cite[Proposition 3]{kwasnicki2012eigenvalues}.

\begin{lem}\label{lemm}
Let $\frac 12<s<1$, $\Omega=(-1,1)$ and $(\rho_n)_{n\in\NN}$ the eigenvalues of $(-d_x^2)^s_D$. Then the following assertions hold.

\begin{enumerate}
\item[(a)] The eigenvalues $(\rho_n)_{n\geq 1}$ are simple.

\item[(b)] There is a constant $\gamma=\gamma(s)\ge \frac{\pi}{2}$ such that for $n$ large enough,
\begin{align}\label{Gap}
\Big(\rho_{n+1}^{\frac{1}{2s}}-\rho_n^{\frac{1}{2s}}\Big)\ge \gamma.
\end{align}
\end{enumerate}
\end{lem}

\subsection{Some well-posedness results}

Here we will not consider the case $M=0$ which corresponds to a fractional wave equation without memory for which the well posedness may be easily obtained through classical semigroup techniques. We have the following existence and continuous dependence result.

\begin{thm}\label{wp_thm}
Let $0<s<1$. For any $(y^0,y^1)\in\Hs{1}\times L^2(-1,1)$ and $u\in L^2(\O)$, the system \eqref{wave_mem} has a unique solution $y\in C([0,T];\Hs{1})\cap C^1([0,T];L^2(-1,1))$. Moreover, there is a constant $\mathcal C>0$ (depending only on $T$) such that
\begin{align}\label{norm_est}
    \norm{y}{C([0,T];\Hs{1})\cap C^1([0,T];L^2(-1,1))}\leq \mathcal C\left(\norm{(y^0,y^1)}{\Hs{1}\times L^2(-1,1)}+\norm{u}{L^2(\O)}\right).
\end{align}
\end{thm}

\begin{proof}
Our proof is based on a standard fixed point argument. Let 
\begin{align*}
    \zed:= C([0,T];\Hs{1})\cap C^1([0,T];L^2(-1,1))
\end{align*}
be endowed with the  norm
\begin{align*}
    \norm{y}{\zed}:=\left(\norm{e^{-\alpha t}y}{C([0,T];\Hs{1})}^2 + \norm{e^{-\alpha t}y}{C^1([0,T];L^2(-1,1))}^2\right)^{\frac 12},
\end{align*}
where $\alpha$ is a positive real number whose value will be given below. Clearly,
\begin{align*}
    e^{-\alpha T}\norm{y}{C([0,T];\Hs{1})\cap C^1([0,T];L^2(-1,1))}\leq\norm{y}{\zed}\leq \norm{y}{C([0,T];\Hs{1})\cap C^1([0,T];L^2(-1,1))}.
\end{align*}
Therefore, $(\zed,\norm{\cdot}{\zed})$ is a Banach space equivalent to $C([0,T];\Hs{1})\cap C^1([0,T];L^2(-1,1))$. 

{\bf Step 1}. Define the map
\begin{align*}
    \mathcal F:\zed\rightarrow \zed,\;\;
    \widetilde{y}\mapsto \widehat{y},
\end{align*}
where $\widehat{y}$ is the solution to \eqref{wave_mem} with $\int_0^t \ffl{s}{y}(\tau)\,d\tau$ being replaced by $\int_0^t \ffl{s}{\widetilde{y}}(\tau)\,d\tau$. That is,
\begin{align*}
	\begin{cases}
		\displaystyle \widehat{y}_{tt}(t,x)+\ffl{s}{\widehat{y}}(t,x) = M\int_0^t \ffl{s}{\widetilde{y}}(\tau,x)\,d\tau +  \mathbf{1}_{\omega(t)}u(t,x) & (t,x)\in Q,
		\\
		\widehat{y}(t,x)=0 & (t,x)\in Q^c,
		\\
		\widehat{y}(0,x)=y^0(x),\;\;\widehat{y}_t(0,x)=y^1(x) & x\in(-1,1).
	\end{cases}
\end{align*}

We claim that  $\mathcal F(\zed)\subset\zed$. Indeed, using the well-posedness results for the wave equation with non-homogeneous terms we can deduce that 
\begin{align}\label{ee}
    \norm{\widehat{y}}{\zed} &\leq \norm{\widehat{y}}{C([0,T];\Hs{1})\cap C^1([0,T];L^2(-1,1))}
   \notag \\
    &\leq \mathcal C\left[\norm{(y^0,y^1)}{\Hs{1}\times L^2(-1,1)} + \norm{u}{L^2(\O)} + \norm{M\int_0^t \ffl{s}{\,\widetilde{y}}(\tau)\,d\tau}{L^2((0,T);\Hs{-1})}\right]
   \notag \\
    &\leq \mathcal C\left[\norm{(y^0,y^1)}{\Hs{1}\times L^2(-1,1)} + \norm{u}{L^2(\O)} + |M|\left(\int_0^T\norm{\ffl{s}{\,\widetilde{y}}(t)}{\Hs{-1}}^2\,dt\right)^{\frac 12}\right]
  \notag  \\
    &= \mathcal C\left[\norm{(y^0,y^1)}{\Hs{1}\times L^2(-1,1)} + \norm{u}{L^2(\O)} + |M|\norm{\widetilde{y}}{L^2((0,T);\Hs{1})}\right],
\end{align}
where $\mathcal C$ is a positive constant depending only on $T$. Hence, $\mathcal F(\zed)\subset\zed$ and the claim is proved. 

{\bf Step 2}. Given $\ttilde{y}\in\zed$, let $\widehat{\widehat{y}}=\mathcal F(\,\ttilde{y}\,)$ be the corresponding solution to \eqref{wave_mem}. Using \eqref{ee} we have that

\begin{align*}
    e^{-\alpha t}& \left(\norm{\mathcal F(\,\widetilde{y}\,)(t)-\mathcal F(\,\ttilde{y}\,)(t)}{\Hs{1}} + \norm{\mathcal F(\,\widetilde{y}\,)_t(t)-\mathcal F(\,\ttilde{y}\,)_t(t)}{L^2(-1,1)}\right)
    \\
    &\leq e^{-\alpha t}\norm{\widehat{y}-\widehat{\widehat{y}}}{C([0,t];\Hs{1})\cap C^1([0,t];L^2(-1,1))}
    \\
    &\leq \mathcal C|M|e^{-\alpha t}\left(\int_0^t \norm{\widetilde{y}(\tau)-\ttilde{y}(\tau)}{\Hs{1}}^2\,d\tau\right)^{\frac 12}
    \\
    &= \mathcal C|M|\left(\int_0^t e^{-2\alpha (t-\tau)}\norm{e^{-\alpha \tau}\left(\widetilde{y}(\tau)-\ttilde{y}(\tau)\right)}{\Hs{1}}^2\,d\tau\right)^{\frac 12}
    \\
    &\leq \mathcal C|M|\left(\int_0^t e^{-2\alpha (t-\tau)}\,d\tau\right)^{\frac 12} \norm{\widetilde{y}-\ttilde{y}}{\zed} = \mathcal C|M|\left(\frac{1-e^{-2\alpha t}}{2\alpha}\right)^{\frac 12} \norm{\widetilde{y}-\ttilde{y}}{\zed}.
\end{align*}
This implies that

\begin{align*}
    \norm{\mathcal F(\,\widetilde{y}\,)-\mathcal F(\,\ttilde{y}\,)}{\zed} \leq \frac{\mathcal C|M|}{\sqrt{2\alpha}} \norm{\widetilde{y}-\ttilde{y}}{\zed}.
\end{align*}
Hence, taking $\alpha=2\mathcal C^2M^2$ we obtain
\begin{align*}
    \norm{\mathcal F(\,\widetilde{y}\,)-\mathcal F(\,\ttilde{y}\,)}{\zed} \leq \frac 12 \norm{\widetilde{y}-\ttilde{y}}{\zed}.
\end{align*}
We have shown that $\mathcal F$ is a contraction. 

{\bf Step 3}. Since  $\mathcal F$ is a contraction, it has a unique fixed point which is the solution to \eqref{wave_mem}. Let now $y$ be this unique solution. Then
\begin{align*}
	\norm{y(t)}{\Hs{1}}^2 &+ \norm{y_t(t)}{L^2(-1,1)}^2 
	\\
	\leq&\, \mathcal C\left(\norm{(y^0,y^1)}{\Hs{1}\times L^2(-1,1)}^2 + \norm{u}{L^2(\O)}^2 + |M|\int_0^t\norm{y(\tau)}{\Hs{1}}^2\,d\tau\right)
	\\
	\leq&\, \mathcal C\left(\norm{(y^0,y^1)}{\Hs{1}\times L^2(-1,1)}^2 + \norm{u}{L^2(\O)}^2\right) + \mathcal C|M|\int_0^t\left(\norm{y(\tau)}{\Hs{1}}^2+\norm{y_\tau(\tau)}{L^2(-1,1)}^2\,d\tau\right).
\end{align*}
Thus, using Gronwall's inequality we obtain that
\begin{align*}
    \norm{y(t)}{\Hs{1}}^2 + \norm{y_t(t)}{L^2(-1,1)}^2 \leq \mathcal C\left(1 + \mathcal C|M|e^{\mathcal C|M|t}\right)\left(\norm{(y^0,y^1)}{\Hs{1}\times L^2(-1,1)}^2 + \norm{u}{L^2(\O)}^2\right).
\end{align*}
We have shown \eqref{norm_est} and  the proof is finished.
\end{proof}

\subsection{The main result}
In this section we state the main result of the article. We start by recalling that in the setting of problems with memory terms, the classical notion of controllability which requires that $y(T,x)=y_t(T,x)=0$, is not completely accurate. Indeed, in order to guarantee that the dynamics can reach the equilibrium, also the memory term has to be taken into account. In particular, it is necessary that also the memory reaches the null value, that is,
\begin{align}\label{eq:in0}
    \int_0^T \ffl{s}{y}(\tau,x)\,d\tau = 0.
\end{align}

If instead, we do not pay attention to turn off the accumulated memory, i.e. if \eqref{eq:in0} does not hold, then the solution $y$ will not stay at the rest after time $T$ as $t$ evolves. Hence, the correct notion of controllability in this framework is given by the following definition (see e.g., \cite[Definition 1.1]{lu2017null}).

\begin{defn}\label{control_def}
Given $\sigma\ge 0$, the system \eqref{wave_mem} is said to be memory-type null controllable at time $T$ if for any couple of initial data $(y^0,y^1)\in\Hs{\sg+1}\times\Hs{\sg}$, there exits a control $u\in L^2(\mathcal O)$ such that the corresponding solution $y$ satisfies
\begin{align}\label{mem_control}
    y(T,x) = y_t(T,x) = \int_0^T \ffl{\tau}{y}(\tau,x)\,d\tau = 0,\qquad x\in(-1,1).
\end{align}
\end{defn}

The main result of the present work is the following theorem.

\begin{thm}\label{control_thm}
Let $\frac 12<s<1$, $\gamma>0$ be given by \eqref{Gap}, $c\in\RR\setminus\{-\gamma,0,\gamma\}$, $T>2\pi\left(\frac{1}{|c|}+\frac{1}{|c+\gamma|}+\frac{1}{|c-\gamma|}\right)$, $\omega_0$ a non-empty open set in $(-1,1)$ and
\begin{align*}
    \omega(t) = \omega_0-ct,\qquad t\in [0,T].
\end{align*}
Then for each  $(y^0,y^1)\in \Hs{3}\times\Hs{2}$, there exists a control $u\in L^2({\mathcal O})$ such that the solution $y$ of \eqref{wave_mem} satisfies \eqref{mem_control}.
\end{thm}

The proof of Theorem \ref{control_thm} will be given in Section \ref{control_sect}. It will be based on the moment method and a careful analysis of the spectrum of the operator associated with the system.

The rest of the paper is organized as follows. 
In Section \ref{cv_sect}, we give a characterization of the control problem through the adjoint equation associated to \eqref{wave_mem}. Section \ref{spectrum_sect} is devoted to a complete spectral analysis for the problem which will then be fundamental in the construction of a biorthogonal sequence in Section \ref{bio_sec} and the resolution of the moment problem. Finally, in Section \ref{control_sect} we give the proof of our controllability result.

\section{Intermediate results}\label{cv_sect}

Here, we give a characterization of the control problem by means of the adjoint system associated with \eqref{wave_mem}. Using a simple integration by parts we have that the following system
\begin{align}\label{wave_mem_adj}
    \begin{cases}
        \displaystyle p_{tt}(t,x)+\ffl{s}{p}(t,x) - M\int_t^T \ffl{s}{p}(\tau,x)\,d\tau - M\ffl{s}{q^0}(x) = 0 & (t,x)\in Q,
        \\
        p(t,x)=0 & (t,x)\in Q^c,
        \\
        p(T,x)=p^0(x),\;\;p_t(T,x)=p^1(x) & x\in(-1,1),
    \end{cases}
\end{align}
can be viewed as the adjoint system associated with \eqref{wave_mem}.

Let us mention that the term $M\ffl{s}{q^0}(x)$ takes into account the presence of the memory in \eqref{wave_mem} and the fact that according to Definition \ref{control_def}, the controllability of our original equation is really reached only if also this memory is driven to zero. We have the following result.

\begin{lem}\label{control_id_lemma}
Let $0<s<1$. Then the following assertions are equivalent.
\begin{enumerate}
\item[(i)]The system \eqref{wave_mem} is memory-type null controllable in time $T$.

 \item[(ii)] For each initial data $(y^0,y^1)\in\Hs{1}\times L^2(-1,1)$, there exits a control $u\in L^2(\mathcal O)$ such that 
\begin{align}\label{control_id}
    \int_0^T\int_{\omega(t)} u(t,x)\bar{p}(t,x)\,dxdt = \big\langle y^0(\cdot),p_t(0,\cdot)\big\rangle_{\hs{1}(-1,1),\hs{-1}(-1,1)} - \int_{-1}^1 y^1(x)\bar{p}(0,x)\,dx,
\end{align}
for any $(p^0,p^1,q^0)\in L^2(-1,1)\times\Hs{-1}\times \Hs{1}$, where $p$ is the unique solution to  \eqref{wave_mem_adj}.
\end{enumerate}
\end{lem}

\begin{proof}
Firstly, we multiply \eqref{wave_mem} by $\bar{p}$ and we integrate by parts over $Q$. Taking into account the exterior condition, we get
\begin{align*}
    \int_0^T\int_{\omega(t)} u(t,x)\bar{p}(t,x)\,dxdt =&\, \int_Q\left(y_{tt}(t,x)+\ffl{s}{y}(t,x) - M\int_0^t \ffl{s}{y}(\tau,x)\,d\tau\right)\bar{p}(t,x)\,dxdt
    \\
    =&\,\int_{-1}^1 \Big(y_t(t,x)\bar{p}(t,x)-y(t,x)\bar{p}_t(t,x)\Big)\Big|_0^T\,dx 
    \\
    &+ \int_Q y(t,x)\Big(\bar{p}_{tt}(t,x)+\ffl{s}{\bar{p}}(t,x)\Big)\,dxdt
    \\
    &- M\int_0^T\int_{-1}^1\left(\int_0^t \ffl{s}{y}(\tau,x)\,d\tau\right)\bar{p}(t,x)\,dxdt.
\end{align*}

Secondly,  we have that
\begin{align*}
    M&\int_0^T\int_{-1}^1 \left(\int_0^t \ffl{s}{y}(\tau,x)\,d\tau\right)\bar{p}(t,x)\,dxdt
    \\
    &= -M\int_0^T\int_{-1}^1\left(\int_0^t \!\ffl{s}{y}(\tau,x)\,d\tau\right)\frac{d}{dt}\left(\int_t^T\! \bar{p}(\tau,x)\,d\tau + \bar{q}^{\,0}(x)\right)\,dxdt
    \\
    &= -M\!\int_{-1}^1\! \left(\int_0^T\!\! \ffl{s}{y}(\tau,x)\,d\tau\!\right)\!\bar{q}^{\,0}(x)\,dx + M\!\int_Q y(t,x)\!\left(\int_t^T\!\! \ffl{s}{\bar{p}}(\tau,x)\,d\tau + \ffl{s}{\bar{q}^{\,0}}(x)\!\right)\,dxdt.
\end{align*}

Thirdly, taking into account the regularity properties of the solution $y$, we obtain
\begin{align}\label{control_id_prel}
	\int_0^T \int_{\omega(t)} u(t,x)\bar{p}(t,x)\,dxdt =&\,\int_{-1}^1 y_t(T,x)\bar{p}^{\,0}(x)\,dx -\big\langle y(T,\cdot),p^1(\cdot)\big\rangle_{\hs{1}(-1,1),\hs{-1}(-1,1)} \notag 
	\\
	& - \int_{-1}^1 y^1(x)\bar{p}(0,x)\,dx + \big\langle y^0(\cdot),p_t(0,\cdot)\big\rangle_{\hs{1}(-1,1),\hs{-1}(-1,1)} \notag
	\\
	&- M\int_0^T \left\langle \ffl{s}{y}(t,\cdot),q^0(\cdot)\right\rangle_{\hs{-1}(-1,1),\hs{1}(-1,1)}\,dt.
\end{align}

Now assume that (i) holds. Then \eqref{control_id} follows from \eqref{control_id_prel} and \eqref{mem_control}. We have shown that (i) implies (ii). Next assume that \eqref{control_id} holds. Then, \eqref{control_id_prel} can be rewritten as
\begin{align*}
    \int_{-1}^1 y_t(T,x)\bar{p}^{\,0}(x)\,dx -\big\langle y(T,\cdot),p^1(\cdot)\big\rangle_{\hs{1}(-1,1),\hs{-1}(-1,1)} -M\int_0^T \left\langle \ffl{s}{y}(t,\cdot),q^0(\cdot)\right\rangle_{\hs{-1}(-1,1),\hs{1}(-1,1)}\,dt=0.
\end{align*}
Since the above identity is verified  for all $(p^0,p^1,q^0)\in L^2(-1,1)\times\Hs{-1}\times  \Hs{1}$, we can immediately deduce that \eqref{mem_control} holds.  We have shown that (ii) implies (i) and the proof is finished.
\end{proof}

Let us notice that we shall prove in Lemma \ref{sol_adj_lemma} that  the system \eqref{wave_mem_adj} has a solution in the energy space associated with the initial data. 

We also observe that the systems \eqref{wave_mem} and \eqref{wave_mem_adj} can be rewritten as the following systems coupling a PDE and an ODE. More precisely, \eqref{wave_mem} and \eqref{wave_mem_adj} are equivalent to

\begin{align}\label{wave_mem_syst}
    \begin{cases}
        y_{tt}(t,x) +\ffl{s}{y}(t,x) - Mz(t,x) = \mathbf{1}_{\omega(t)}u(t,x) &(t,x)\in Q,
        \\
        z_t(t,x) = \ffl{s}{y}(t,x) &(t,x)\in Q,
        \\
        y(t,x)=0 &(t,x)\in Q^c
        \\
        y(0,x) = y^0(x),\;\; y_t(0,x) = y^1(x) & x\in(-1,1),
        \\
        z(0,x) = 0 & x\in(-1,1),
    \end{cases}
\end{align}
and
\begin{align}\label{wave_mem_adj_syst}
    \begin{cases}
        p_{tt}(t,x) +\ffl{s}{p}(t,x) - M\ffl{s}{q}(t,x) = 0 &(t,x)\in Q,
        \\
        -q_t(t,x) = p(t,x) &(t,x)\in Q,
        \\
        p(t,x)=0&(t,x)\in Q^c,
        \\
        p(T,x) = p^0(x),\;\; y_t(T,x) = p^1(x), & x\in(-1,1),
        \\
        q(T,x) = q^0(x) & x\in(-1,1),
    \end{cases}
\end{align}
respectively.

In order to prove our controllability result, we  introduce the following change of variables:

\begin{align}\label{cv}
	x\mapsto x':=x+ct ,
\end{align}
where the parameter $c$ is a constant velocity which belongs to $\mathbb{R}\setminus\{-\gamma,0,\gamma\}$ and $\gamma>0$ is a suitable given constant that we shall specify later. Let

\begin{align*}
	\mathcal I:=(-1+ct,1+ct),\;\;\;\mathcal Q:=\mathcal I\times(0,T),\;\;\; \mathcal Q^c:=\mathcal I^c\times(0,T),
\end{align*}
and
\begin{align*}
	\xi(t,x'):=y(t,x),\;\;\; \zeta(t,x'):=z(t,x),\;\;\; \varphi(t,x'):=p(t,x),\;\;\; \psi(t,x'):=q(t,x).
\end{align*}

A simple computation shows that, from \eqref{wave_mem_syst}, \eqref{wave_mem_adj_syst},  \eqref{cv} and after having denoted $x'=x$ (with some abuse of notation), we obtain 
\begin{align}\label{wave_mem_syst_cv}
    \begin{cases}
        \xi_{tt}(t,x) + c^2\xi_{xx}(t,x) + \ffl{s}{\xi}(t,x) + 2c\xi_{xt}(t,x) - M\zeta(t,x) = \mathbf{1}_{\omega_0}\tilde{u}(t,x) &(t,x)\in\Q,
        \\
        \zeta_t(t,x) + c\zeta_x(t,x) = \ffl{s}{\xi}(t,x) &(t,x)\in\Q,
        \\
        \xi(t,x)=0 & (t,x)\in\Q^c,
        \\
        \xi(0,x) = y^0(x),\;\; \xi_t(0,x) = y^1(x)-cy^0_x(x) & x\in\I,
        \\
        \zeta(0,x) = 0 & x\in\I,
\end{cases}
\end{align}
and
\begin{align}\label{wave_mem_adj_syst_cv}
    \begin{cases}
        \varphi_{tt}(t,x) +c^2\varphi_{xx}(t,x) + \ffl{s}{\varphi}(t,x) + 2c\varphi_{xt}(t,x) - M\ffl{s}{\psi}(t,x) = 0 &(t,x)\in\Q,
        \\
        -\psi_t(t,x) - c\psi_x(t,x) = \varphi(t,x) &(t,x)\in\Q,
        \\
        \varphi(t,x)=0 & (t,x)\in\Q^c,
        \\
        \varphi(T,x) = p^0(x),\;\; \varphi_t(T,x) = p^1(x)-cp^0_x(x) & x\in\I,
        \\
        \psi(T,x) = q^0(x) & x\in\I,
    \end{cases}
\end{align}
respectively.

We remark that, after the above change of variables, our new systems are defined in a domain which is moving with constant velocity $c$. On the other hand, in this moving framework the control region is now fixed (that is, it does not depend on the time variable). 

\begin{rem}
{\em 
Contrary to the case of the Laplace operator ($s=1$) investigated in \cite{biccari2018null}, in the case of the fractional Laplacian studied here, after using the change of variable \eqref{cv}, the new obtained systems \eqref{wave_mem_syst_cv} and \eqref{wave_mem_adj_syst_cv} are of higher order in the space variable $x$. More precisely, we started with the fractional Laplace operator and after the change of variable our new systems include the Laplace operator. This will require more regularity on the initial data in order to obtain some well-posedness results. For this this reason, we need to introduce some new function spaces.  }
\end{rem}

Let $(-d_x^2)_D$ be the realization in $L^2(\I)$ of the Laplace operator $-d_x^2$ with the zero Dirichlet boundary condition. It is well-known that $D((-d_x^2)_D)=H^2(\I)\cap H_0^1(\I)\subset\mathbb H_s^1(\I)=D((-d_x^2)_D^s)$. For $\sigma\ge 0$, we denote by $\mathbb H_1^\sigma (\I)$ the domain of the $\sigma$-power of $(-d_x^2)_D$, that is, $D((-d_x^2)_D^\sigma)$. We shall also denote by $\mathbb H_1^{-\sigma }(\I)$ the dual of $\mathbb H_1^\sigma (\I)$ with respect to the pivot space $L^2(\I)$. Then it is easy to see that $\mathbb H_1^\sigma (\I)\subseteq \mathbb H_s^\sigma (\I)$ for every $\sigma\ge 0$, where we recall that $ \mathbb H_s^\sigma (\I)$ has been introduced in \eqref{Hsigma}. Hence, we have the following continuous embeddings:

\begin{align}\label{Cont-Emb}
\mathbb H_1^\sigma (\I)\hookrightarrow \mathbb H_s^\sigma (\I)\hookrightarrow L^2(\I)\hookrightarrow \mathbb H_s^{-\sigma}(\I)\hookrightarrow \mathbb H_1^{-\sigma}(\I).
\end{align}

Next, in analogy to Lemma \ref{control_id_lemma} we have the following characterization of our control problem.

\begin{lem}\label{control_id_lemma_syst}
Let $0<s<1$. Then the following assertions are equivalent.
\begin{enumerate}
\item [(i)] The system \eqref{wave_mem_syst_cv} is memory-type null controllable in time $T$.

\item [(ii)] For each initial data $(y^0,y^1)\in\mathbb H_s^1(\I)\times L^2(\I)$, there exits a control $\tilde{u}\in L^2((0,T)\times \omega_0)$ such that
\begin{align}\label{control_id_cv}
    \int_0^T\int_{\omega_0} u(t,x)\bar{\varphi}(t,x)\,dxdt= \big\langle y^0(\cdot),\varphi_t(0,\cdot)\big\rangle_{\mathbb H_s^1(\I),\mathbb H_s^{-1}(\I)} - \int_{\I} \big(y^1(x)+cy^0_x(x)\big)\bar{\varphi}(0,x)\,dx,
\end{align}
for any $(p^0,p^1,q^0)\in L^2(\I)\times\mathbb H_s^{-1}(\I)\times \mathbb H_s^1(\I)$, where $(\varphi,\psi)$ is the unique solution to \eqref{wave_mem_adj_syst_cv}.
\end{enumerate}
\end{lem}

\begin{proof}
The proof follows the lines of the proof  of Lemma \ref{control_id_lemma}. We omit it for brevity.
\end{proof}

\section{Analysis of the spectrum of the operator associated to the systems}\label{spectrum_sect}

In this section we make the spectral analysis of the operators associated to our adjoint systems. This analysis will be fundamental in the proof of the controllability result. 


\subsection{Spectral analysis of the operator associated with the system \eqref{wave_mem_adj_syst}}

We observe that the system \eqref{wave_mem_adj_syst} can be rewritten as  the following first order Cauchy problem

\begin{align}\label{eq:memsys}
    \begin{cases}
        Y'(t) + \mathcal A Y = 0 & t\in(0,T),
        \\
        Y(0)=Y^0,
    \end{cases}
\end{align}
where $Y=\left(\begin{array}{c}
        p
        \\
        r
        \\
        q\end{array}\right)$,
$Y^0=\left(\begin{array}{c}
    p^0
    \\
    r^0
    \\
    q^0\end{array}\right)$ and the unbounded operator ${\mathcal A}: D({\mathcal A})\rightarrow
X_{-\sigma}$ is defined by
\begin{align*}
\mathcal A\left(\begin{array}{c}
    p
    \\
    r
    \\
    q\end{array}\right)=\left(\begin{array}{c}
        -r
        \\
        \ffl{s}{p}-M\ffl{s}{q}
        \\
        p\end{array}\right).
\end{align*}
The spaces $X_{-\sigma}$ (for $\sigma\ge 0)$ and $D({\mathcal A})$ are given  by
\begin{align}\label{eq:xsmi}
	&X_{-\sigma}=\Hs{-\sg}\times \Hs{-\sg-1}\times \Hs{-\sg},
	\\ \label{eq:dsmi}
	&D({\mathcal A})=\Hs{-\sg+1}\times \Hs{-\sg}\times \Hs{-\sg+1}.
\end{align}

\begin{thm}\label{eigen_mu_thm}
Let $0<s<1$ and $\sigma\ge 0$.
The spectrum of the operator $(D({\mathcal A}),\mathcal A)$ is given by
\begin{align}\label{eq:spec}
    \operatorname{Spec}(\mathcal A)=\left(\mu_n^j\right)_{n\geq 1,\, 1\leq j\leq 3}\cup\{M\},
\end{align}
where the eigenvalues $\mu_n^j$ verify the following:
\begin{align}\label{eq:eig}
    \begin{cases}
        \displaystyle \mu_n^1 = M - \frac{M^3}{\rho_n} + \mathcal O\left(\frac{1}{n^4}\right) & n\geq 1,
        \\[15pt]
        \displaystyle \mu_n^2 = -\frac{\mu_n^1}{2} + i\,\sqrt{3\left(\frac{\mu_n^1}{2}\right)^2+\rho_n} & n\geq 1,
        \\[15pt]
        \displaystyle \mu_n^3 = \bar{\mu}^2_n & n\geq 1.
    \end{cases}
\end{align}
Each eigenvalue $\mu_n^j\in \operatorname{Spec}({\mathcal A})$ is double and has two associated eigenvectors given by
\begin{align}\label{eq:eigf}
    \Phi_{\pm n}^j = \left(\begin{array}{c}
        1
        \\
        -\mu_n^j
        \\ [5pt]
        \displaystyle\frac{1}{\mu_n^j}\end{array}\right) e^{\pm i\rho_n^{\frac{1}{2s}}x}, \qquad j\in\{1,2,3\},\,\, n\geq 1.
    \end{align}
\end{thm}

\begin{proof}
We prove the result in several steps.

{\bf Step 1}. From the equality
\begin{align*}
    \mathcal A\left(\begin{array}{c}
        \phi^1
        \\
        \phi^2
        \\
        \phi^3
    \end{array}\right) = \left(\begin{array}{c}
        -\phi^2
        \\
        \ffl{s}{\phi^1} - M\ffl{s}{\phi^3}
        \\
        \phi^1
    \end{array}\right) =\mu\left(\begin{array}{c}
        \phi^1
        \\
        \phi^2
        \\
        \phi^3
    \end{array}\right),
\end{align*}
we can deduce that
\begin{align*}
    \begin{cases}
        \phi^2 = -\mu\phi^1,
        \\
        \phi^3 = \mu^{-1}\phi^1,
        \\
        (\mu-M)\ffl{s}{\phi^1} = -\mu^3\phi^1.
    \end{cases}
\end{align*}
Plugging the first two equations from the above system in the third one, we immediately get
\begin{align*}
    \ffl{s}{\phi^1} = \left(-\frac{\mu^3}{\mu-M}\right)\phi^1.
\end{align*}
Consequently, $\mu$ is an eigenvalue of  $\mathcal A$  if and only if it verifies the characteristic equation
\begin{align}\label{eq:lambda}
	K(\mu):=\mu^3+\rho_n\mu-M\rho_n = 0, \quad n\geq 1.
\end{align}

{\bf Step 2}. We claim that for any $\kappa\in\RR$, 
\begin{align}\label{claim}
\ffl{s}{e^{i\kappa x}}=|\kappa|^{2s}e^{i\kappa x}.
\end{align}
Indeed, by definition of the fractional Laplacian (see \eqref{fl}) we have
\begin{align*}
	\ffl{s}{e^{i\kappa x}} = C_s\, \mbox{P.V.}\,\int_{\RR} \frac{e^{i\kappa x}-e^{i\kappa y}}{|x-y|^{1+2s}}\,dy = C_s\, e^{i\kappa x} \mbox{P.V.}\,\int_{\RR} \frac{1-e^{i\kappa(y-x)}}{|x-y|^{1+2s}}\,dy.
\end{align*}

Now, applying the change of variables $z=\kappa(y-x)$, using the definition of principal value and the expression for the constant $C_s$ given in \cite[Section 3]{dihitchhiker}, we get 

\begin{align*}
	\ffl{s}{e^{i\kappa x}} &= C_s|\kappa|^{2s}\, e^{i\kappa x} \mbox{P.V.}\,\int_{\RR} \frac{1-e^{iz}}{|z|^{1+2s}}\,dz = C_s|\kappa|^{2s}\, e^{i\kappa x}\lim_{\varepsilon\to 0^+}\,\int_{|z|>\varepsilon} \frac{1-e^{iz}}{|z|^{1+2s}}\,dz 
	\\
	&= C_s|\kappa|^{2s}\, e^{i\kappa x}\lim_{\varepsilon\to 0^+}\,\int_{\varepsilon}^{+\infty} \frac{2-2\cos(z)}{z^{1+2s}}\,dz = C_s|\kappa|^{2s}\, e^{i\kappa x}\,\int_{\RR} \frac{1-\cos(z)}{|z|^{1+2s}}\,dz 
	\\
	&= C_s|\kappa|^{2s}\, e^{i\kappa x}C_s^{-1} = |\kappa|^{2s}\, e^{i\kappa x},
\end{align*}
and we have shown the claim \eqref{claim}.

{\bf Step 3}. Using \eqref{claim} we can readily check that $\phi^1$ takes the form
\begin{align*}
    \phi^1(x) = e^{i\rho_n^{\frac{1}{2s}}x},\;\;\; n\geq 1.
\end{align*}

Since $-M\rho_n$ and $M^3$ have opposite signs and $K(0)=-M\rho_n$, $K(M)=M^3$ for each $n\ge 1$, it follows that for each $n\geq 1$, \eqref{eq:lambda} has a unique real root located between 0 and $M$. This real root will be denoted by $\mu^1_n$. 
Since $\mu_n^1$ verifies
\begin{align*}
    \mu_n^1 = M-\frac{M(\mu_n^1)^2}{(\mu_n^1)^2+\rho_n}, \quad n\geq 1,
\end{align*}
we have that
\begin{align}\label{eq:lambdanp}
    \mu_n^1 = M - \frac{M^3}{\rho_n} + \mathcal O\left(\frac{1}{n^4}\right), \qquad n\geq 1.
\end{align}
Thus, in particular $M$ is an accumulation point in the spectrum and belongs to the essential spectrum. 

{\bf Step 4}. Finally, we analyze the complex roots of the characteristic equation \eqref{eq:lambda}. Let us write $\mu=\alpha+i\beta$ with $\alpha,\beta\in\mathbb{R}$ and $\beta\neq 0$. Plugging this value in \eqref{eq:lambda} and setting both real and imaginary parts to be equal to zero, we get that 

\begin{align*}
    \begin{cases}
        \beta^2=3\alpha^2+\rho_n,
        \\
        -8\alpha^3-2\alpha \rho_n-M\rho_n=0.
    \end{cases}
\end{align*}
Hence, $-2\alpha$ is the real root of the characteristic equation \eqref{eq:lambda}, and we can deduce that
\begin{align}
    \begin{cases}
        \displaystyle \mu_n^{2,3} = \alpha_n^{2,3}+i\beta_n^{2,3} & n\geq 1,
        \\[7pt]
        \displaystyle \alpha_n^2 = \alpha^3_n = -\frac{\mu_n^1}{2} & n\geq 1,
        \\[7pt]
        \displaystyle \beta_n^{2} = \sqrt{3\left(\frac{\mu_n^1}{2}\right)^2+\rho_n},\quad  \beta_n^3=-\sqrt{3\left(\frac{\mu_n^1}{2}\right)^2+\rho_n} & n\geq 1.
    \end{cases}
\end{align}

We have identified the three families of eigenvalues $\mu_n^j$, $n\geq 1$, $j\in\{1,2,3\}$ which satisfy \eqref{eq:eig}.
To each eigenvalue $\mu_n^j$, $j\in\{1,2,3\}$, $n\geq 1$, they correspond two independent eigenfunctions $\Phi^j_{\pm n}$ given
by the formula \eqref{eq:eigf}.  The proof is finished.
\end{proof}

\begin{rem}
{\em
Notice that taking $M=0$ in \eqref{eq:lambda}, we recover the equation $\mu^3+\rho_n\mu =0$ whose non-trivial solutions are $\mu=\pm i\rho_n^{\frac 12}$, $n\geq 1$. In other words, as it is expected, we recover the eigenvalues of the fractional wave operator without memory.}
\end{rem}

The following result will be used in the analysis of the spectrum of the operator associated with \eqref{wave_mem_adj_syst_cv}. 

\begin{lem}\label{lemma:sir} 
Let $0<s<1$.
The sequence of real numbers $\left(|\mu_n^1|\right)_{n\geq 1}$ verifies the estimates
\begin{equation}\label{eq:red}
	\frac{|M|}{\frac{M^2}{\rho_1}+1}\leq \left|\mu_n^1\right|<|M|, \qquad n\geq 1.
\end{equation}
Moreover, if $\frac 12 <s<1$, then this sequence is increasing.  
\end{lem}

\begin{proof} 
We recall that $|\mu_n^1|\in(0,|M|)$ is the unique real root of the equation 
\begin{equation}\label{eq:mumod} 
	\mu^3+\rho_n\mu-\rho_n|M|=0.
\end{equation}
This implies that
\begin{align*}
	\frac{|\mu_{n}^1|}{|M|}=\frac{\rho_n}{(\mu_{n}^1)^2+\rho_n}\geq \frac{\rho_n}{M^2+\rho_n} \geq \frac{1}{\frac{M^2}{\rho_n}+1}\geq \frac{1}{\frac{M^2}{\rho_1}+1},
\end{align*}
and consequently \eqref{eq:red} holds. Finally, using \eqref{eq:mumod} with $\mu=|\mu_n^1|$ and $\mu=|\mu_{n+1}^1|$ we obtain that
\begin{align*}
	\left(|\mu_{n+1}^1|-|\mu_n^1|\right)\left(|\mu_{n+1}^1|^2+|\mu_{n}^1|^2+|\mu_{n+1}^1||\mu_{n}^1|+\rho_{n+1}\right)
	=\left(\rho_{n+1}-\rho_n\right)\left(|M|-|\mu_n^1|\right).
\end{align*}

By Lemma \ref{lemm}, if $\frac 12<s<1$ the eigenvalues $(\rho_n)_{n\ge 1}$ are simple. This implies that $\rho_{n+1}>\rho_n$ for all $n\geq 1$. This, together with the fact that $|\mu_n^1|\in(0,|M|)$, yields
\begin{align*}
	|\mu_{n+1}^1|>|\mu_n^1|.
\end{align*}
We have shown that the sequence  $\left(|\mu_n^1|\right)_{n\geq 1}$ is increasing and the proof is finished.
\end{proof}

We conclude this section with the following remark.

\begin{rem} 
{\em The spectral properties from Theorem \ref{eigen_mu_thm} has some immediate interesting consequences.
\begin{enumerate}

\item[(a)] Since there are eigenvalues with positive real part, it follows that the operator $\mathcal A$ in \eqref{eq:memsys} is not dissipative. However, since all the real parts of the eigenvalues are uniformly bounded, the well-posedness of the equation is ensured.

\item[(b)] The existence of a family of eigenvalues having an accumulation point implies that our initial equation \eqref{wave_mem} cannot be controlled with a fixed support control $\omega_0$, unless we are in the trivial case $\omega_0=(-1,1)$.

\end{enumerate}
}
\end{rem}

\subsection{Spectral analysis of the operator associated with the system \eqref{wave_mem_adj_syst_cv}}
We now move to the spectral analysis for the system \eqref{wave_mem_adj_syst_cv}. We recall that $c$ is a parameter which belongs to  $\mathbb{R}\setminus\{-\gamma,0,\gamma\}$ where $\gamma$ is given in \eqref{Gap}. 
Similarly to the previous section, the system \eqref{wave_mem_adj_syst_cv} can be equivalently written as a system of order one in the following way:
\begin{align}\label{eq:memchanged}
    \begin{cases}
        W'(t) + \mathcal A_cW=0 & t\in (0,T),
        \\
        W(T)=W^0,
    \end{cases}
\end{align}
where $W=\left(\begin{array}{c}
    \varphi\\\eta\\\psi\end{array}\right)$,
$W^0=\left(\begin{array}{c} \varphi^0\\\eta^0\\\psi^0\end{array}\right)$ and the unbounded operator ${\mathcal A}_c: D({\mathcal A_c})\rightarrow X_{-\sigma}$ (for $\sigma\ge 0$)  is defined by
\begin{align*}
    \mathcal A_c\left(\begin{array}{c}
        \varphi
        \\
        \eta
        \\
        \psi\end{array}\right)=\left(\begin{array}{c}
            -\eta
            \\
            c^2\varphi_{xx}+\ffl{s}{\varphi}+2c\eta_x-M\ffl{s}{\psi}
            \\
            c\psi_x+\varphi\end{array}\right).
\end{align*}

The domain $D({\mathcal A}_c)$ is given by $D({\mathcal A}_c)=\mathbb H_1^{-\sg+1}(\I)\times \mathbb H_1^{-\sg}(\I)\times \mathbb H_s^{-\sg+1}(\I)$ and here the space $X_{-\sigma}$ is given by
$X_{-\sigma}=\mathbb H_1^{-\sigma}(\I)\times\mathbb H_1^{-\sg-1}(\I)\times\mathbb H_s^{-\sigma}(\I)$.

We notice that using the embeddings \eqref{Cont-Emb}, by taking $\sigma$ large enough, $D(\mathcal A_c)$ can be also defined as $D({\mathcal A}_c)=\mathbb H_s^{-\sg+1}(\I)\times \mathbb H_s^{-\sg}(\I)\times \mathbb H_s^{-\sg+1}(\I)$ and $X_{-\sigma}=\mathbb H_s^{-\sigma}(\I)\times\mathbb H_s^{-\sg-1}(\I)\times\mathbb H_s^{-\sigma}(\I)$.
	
The next two theorems give the spectral properties of the operator ${\mathcal A}_c$.

\begin{thm} 
The spectrum of the operator $(D(\mathcal A_c),\mathcal A_c)$ is given by
\begin{align}\label{eq:specac}
    \operatorname{Spec}(\mathcal A_c)=\left(\lambda_n^j\right)_{n\in\mathbb{Z}^*,\, 1\leq j\leq 3},
\end{align}
where the eigenvalues $\lambda_n^j$ are given by
\begin{align}\label{eq:eigneq}
    \lambda_n^j=\mu_{|n|}^j+i\,\sgn(n)c\rho_{|n|}^{\frac{1}{2s}}, \qquad  n\in\mathbb{Z}^*,\,\, 1\leq j\leq 3,
\end{align}
and $\mu_{|n|}^j$ are given in \eqref{eq:eig}.
\end{thm}

\begin{proof}
From the equality
\begin{align*}
    \mathcal A_c\left(\begin{array}{c}
    \phi^1
    \\
    \phi^2
    \\
    \phi^3
    \end{array}\right) =\lambda\left(\begin{array}{c}
    \phi^1
    \\
    \phi^2
    \\
    \phi^3
    \end{array}\right),
\end{align*}
we can deduce that
\begin{align*}
    \begin{cases}
    \phi^2 = -\lambda\phi^1,
    \\
    \lambda\phi^3 = c\phi^3_x+\phi^1,
    \\
    c^2\phi^1_{xx}+\ffl{s}{\phi^1} + 2c\phi^2_x - M\ffl{s}{\phi^3} = \lambda\phi^2.
    \end{cases}
\end{align*}
To find the solutions of the above system we write
\begin{align*}
    \phi^{\,j}(x)=\sum_{n\in\mathbb{Z}^*}a_{nj}e^{i\,\sgn(n)\rho_{|n|}^{\frac{1}{2s}}x},\quad 1\leq j\leq 3.
\end{align*}
Then, assuming for the moment that $n\geq 1$, we can deduce that the coefficients $a_{nj}$ verify
\begin{align*}
    \begin{cases}
        a_{n2}=-\lambda a_{n1},
        \\[15pt]
        \displaystyle a_{n3}=\frac{a_{n1}}{\lambda-ic\rho_{n}^{\frac{1}{2s}}},
        \\[15pt]
        \displaystyle \left[-c^2\rho_{n}^{\frac{1}{s}}+\rho_{n} -2ic\rho_{n}^{\frac{1}{2s}}\lambda -\frac{M\rho_{n}}{\lambda - ic\rho_{n}^{\frac{1}{2s}}}\right]a_{n1}=-\lambda^2a_{n1}.
    \end{cases}
\end{align*}
Therefore, $\lambda$ is an eigenvalues of ${\mathcal A}_c$ if and only if it verifies the equation
\begin{align*}
    (\lambda-ic\rho_{n}^{\frac{1}{2s}})^3+\rho_n(\lambda-ic\rho_{n}^{\frac{1}{2s}})-M^2\rho_n=0.
\end{align*}

If $n\leq -1$, then a similar computation gives the characteristic equation
\begin{align*}
	(\lambda+ic\rho_{|n|}^{\frac{1}{2s}})^3+\rho_{|n|}(\lambda-ic\rho_{|n|}^{\frac{1}{2s}})-M^2\rho_{|n|}=0.
\end{align*}

By taking into account that $\mu_{|n|}^j$ are the zeros of \eqref{eq:lambda}, we obtain that the eigenvalues of ${\mathcal A}_c$ are given by the formula \eqref{eq:eigneq} and the proof is finished.
\end{proof}

\begin{thm} \label{te:lari}
Each eigenvalue $\lambda_n^j\in \operatorname{Spec}(\mathcal A_c)$
has an associated eigenvector of the form
\begin{align}\label{eq:eigfnew}
    \Psi_{n}^{j}=\left(\begin{array}{c}
                    1
                    \\[5pt]
                    -\lambda_n^j
                    \\[5pt]
                    \displaystyle\frac{1}{\lambda_n^j-i\,\sgn(n)c\rho_{|n|}^{\frac{1}{2s}}}\end{array}\right)e^{i\,\sgn(n)\rho_{|n|}^{\frac{1}{2s}}x},\qquad
    j\in\{1,2,3\},\,\, n\in\mathbb{Z}^*.
\end{align}
In addition, for any $\sigma\geq 0$, the set $\left(\rho_{|n|}^\sigma\Psi^{j}_n\right)_{n\in\mathbb{Z}^*,\,1\leq j\leq 3}$  forms a Riesz basis in  $X_{-\sigma}$.
\end{thm}

\begin{proof}
It is easy to see that to each eigenvalue $\lambda_n^j$, $j\in\{1,2,3\}$, $n\in\mathbb{Z}^*$, corresponds an eigenfunction $\Psi^j_{n}$ given
by \eqref{eq:eigfnew}. Let us show that $\left(\rho_{|n|}^\sigma\Psi^{j}_n\right)_{n\in\mathbb{Z}^*,\,1\leq j\leq 3}$  forms a Riesz basis in  $X_{-\sigma}$. We proof the result in several steps.

{\bf Step 1}. Firstly, we remark that $\left(\Psi^{j}_n\right)_{n\in\mathbb{Z}^*,\,1\leq j\leq 3}$ is complete in $X_{-\sigma}$. This follows from the fact that, given an arbitrary element in  $X_{-\sigma}$,
\begin{align*}
    \left(\begin{array}{c}
        y^0
        \\[2pt]
        w^0
        \\[2pt]
        z^0
    \end{array}\right) = \sum_{n\in\mathbb{Z}^*}\left(\begin{array}{c}
            \beta_n^1
            \\[2pt]
            \beta_n^2
            \\[2pt]
            \beta_n^3
        \end{array}\right)\rho_{|n|}^\sigma e^{i\,\sgn(n)\rho_{|n|}^{\frac{1}{2s}}x},
\end{align*}
 there exists a unique sequence $(a_n^{j})_{n\in\mathbb{Z}^*,\,1\leq j\leq 3}$ such that
\begin{align*}
    \left(\begin{array}{c}
        y^0
        \\
        w^0
        \\
        z^0
    \end{array}\right) = \sum_{n\in\mathbb{Z}^*,\,1\leq j\leq 3}a_n^{j}\rho_{|n|}^\sigma\Psi_n^{j}.
\end{align*}
From \eqref{eq:eigfnew}, the preceding identity is equivalent to the system
\begin{align}\label{eq:sysla}
    \begin{cases}
        a_n^{1}+a_n^{2}+a_n^{3}=\beta_n^1 & n\geq 1,
        \\[5pt]
        \lambda_{n}^1a_n^{1}+\lambda_{n}^2 a_n^{2}+\lambda_{n}^3 a_n^{3}=-\beta_n^2 & n\geq 1,
        \\[5pt]
        \displaystyle\frac{a_{n}^{1}}{\lambda_n^1-i\,\sgn(n)c\rho_{|n|}^{\frac{1}{2s}}}+\frac{a_{n}^{2}}{\lambda_n^2-i\,\sgn(n)c\rho_{|n|}^{\frac{1}{2s}}}+\frac{a_{n}^{3}}{\lambda_n^3-i\,\sgn(n)c\rho_{|n|}^{\frac{1}{2s}}}=\beta_{n}^3 & n\geq 1.
    \end{cases}
\end{align}
The matrix associated with the above system is given by 
\begin{align*}
	B=\left(\begin{array}{ccc}
	1& 1& 1
	\\[3pt]
	\lambda_{n}^1 &\lambda_{n}^2 &\lambda_{n}^3
	\\[3pt]
    \displaystyle \frac{1}{\lambda_n^1-i\,\sgn(n)c\rho_{|n|}^{\frac{1}{2s}}} & \displaystyle \frac{1}{\lambda_n^2-i\,\sgn(n)c\rho_{|n|}^{\frac{1}{2s}}} & \displaystyle \frac{1}{\lambda_n^3-i\,\sgn(n)c\rho_{|n|}^{\frac{1}{2s}}}
    \end{array}\right).
\end{align*}
Since, for each $n\ge 1$, the determinant of $B$ is given by 
\begin{align*}
\det(B) = \frac{(\lambda_n^2-\lambda_n^1)(\lambda_n^3-\lambda_n^1)(\lambda_n^3-\lambda_n^2)}{\left(\lambda_n^1-i\,\sgn(n)c\rho_{|n|}^{\frac{1}{2s}}\right)\left(\lambda_n^2-i\,\sgn(n)c\rho_{|n|}^{\frac{1}{2s}}\right)\left(\lambda_n^3-i\,\sgn(n)c\rho_{|n|}^{\frac{1}{2s}}\right)}\neq 0,
\end{align*}
we can deduce that for each $n\geq 1$, the system has a unique solution $\left(a_n^{j}\right)_{1\leq j\leq 3}$.


{\bf Step 2}. We notice that

\begin{align}\label{eq:ri1cv}
	\Bigg\|\sum_{n\in\mathbb{Z}^*,\,1\leq j\leq 3}& a_n^j\rho_{|n|}^\sigma\Psi_n^j\,\Bigg\|_{X_{-\sigma}}^2 
	\\
	=& \sum_{n\in\mathbb{Z}^*}\rho_{|n|}^{2\sigma}\norm{a_n^1\Psi_n^{1}+a_n^2\Psi_n^{2}+a_n^3\Psi_n^{3}}{X_{-\sigma}}^2 \nonumber 
	\\
	=&\, 2 \sum_{n\in\mathbb{Z}^*}\Bigg[\left|a_n^1+a_n^2+a_n^3\right|^2 + \frac{1}{\rho_{|n|}^2}\left|\lambda_{n}^1a_n^1+\lambda_{n}^2a_n^2+ \lambda_{n}^3a_n^3\right|^2 \nonumber 
	\\
	&+\bigg|\frac{a_n^1}{\lambda_{n}^1-i\,\sgn(n)c\rho_{|n|}^{\frac{1}{2s}}}+\frac{a_n^2}{\lambda_{n}^2-i\,\sgn(n)c\rho_{|n|}^{\frac{1}{2s}}}+\frac{a_n^3}{\lambda_{n}^3-i\,\sgn(n)c\rho_{|n|}^{\frac{1}{2s}}}\bigg|^2\,\Bigg]\notag\\
	=& 2\! \sum_{n\in\mathbb{Z}^*}\!\left\|B_n\! \left(\!\begin{array}{l}a_n^1\\a_n^2\\a_n^3\end{array}\!\right)\right\|_2^2, \nonumber 
	\\[10pt]
\end{align}
where the matrix $B_n$ is given by
\begin{align*}
	B_n=\left(\begin{array}{ccc}
        1& 1& 1
        \\[8pt]
        \displaystyle \frac{\lambda_{n}^1}{\rho_{|n|}} &\displaystyle\frac{\lambda_{n}^2}{\rho_{|n|}}&\displaystyle\frac{\lambda_{n}^3}{\rho_{|n|}}
        \\[23pt]
        \displaystyle \frac{1}{\lambda_n^1-i\,\sgn(n)c\rho_{|n|}^{\frac{1}{2s}}} &\displaystyle\frac{1}{\lambda_n^2-i\,\sgn(n)c\rho_{|n|}^{\frac{1}{2s}}} &\displaystyle \frac{1}{\lambda_n^3-i\,\sgn(n)c\rho_{|n|}^{\frac{1}{2s}}}
    \end{array}\right).
\end{align*}

{\bf Step 3}. 
We claim that there are two constants $0<\textfrak{a}_1<\textfrak{a}_2$, independent of $n$, such that 
\begin{equation}\label{eq:specbn2}
	\sigma(B_n^*B_n)\subset [\textfrak{a}_1,\textfrak{a}_2],\qquad n\in\mathbb{Z}^*.
\end{equation}
We notice that $B_n$ is not singular. Indeed, we have 
\begin{align}
	\det(B_n)=\frac{(\lambda_n^2-\lambda_n^1)(\lambda_n^3-\lambda_n^1)(\lambda_n^3-\lambda_n^2)}{\rho_{|n|}^3\left(\lambda_n^1-i\,\sgn(n)c\rho_{|n|}^{\frac{1}{2s}}\right)\left(\lambda_n^2-i\,\sgn(n)c\rho_{|n|}^{\frac{1}{2s}}\right)\left(\lambda_n^3-i\,\sgn(n)c\rho_{|n|}^{\frac{1}{2s}}\right)}\neq 0,
\end{align}
which implies that, for each $n\in\mathbb{Z}^*$,
\begin{equation}\label{eq:inspei2}
	\min\Big\{|z|\,:\, z\in\sigma(B_n^*B_n)\Big\}>0.
\end{equation}
On the other hand, we have that $B_n^*B_n : = \left(B_{k,\ell}\right)_{1\leq k,\ell\leq 3}$ with
\begin{align*}
	&B_{k,k} = 1+\frac{|\lambda_n^k|^2}{\rho_{|n|}}+\frac{1}{\left|\lambda_n^k-i\,\sgn(n)c\rho_{|n|}^{\frac{1}{2s}}\right|^2},\;\;\; k=1,2,3
	\\
	&B_{k,\ell} = 1+\frac{\overline{\lambda}_n^k\lambda_n^\ell}{\rho_{|n|}}+\frac{1}{\left(\overline{\lambda}_n^k+i\,\sgn(n)c\rho_{|n|}^{\frac{1}{2s}}\right)\left(\lambda_n^\ell-i\,\sgn(n)c\rho_{|n|}^{\frac{1}{2s}}\right)} = \overline{B}_{\ell,k}, \;\;\; k\neq\ell=1,2,3.
\end{align*}	
The above relation, \eqref{eq:eigneq}, and \eqref{eq:eig} imply that
\begin{equation}\label{eq:limbn}
	B_n^*B_n \longrightarrow \widetilde{B}:=\left(\begin{array}{ccc}
	\displaystyle 1+c^2+\frac{1}{|M|^2}& 1+c(c+1)& 1+c(c-1)
	\\[4pt]
	\displaystyle 1+c(c+1)& 1+(c+1)^2 &1+(c+1)(c-1)
	\\[4pt]
	\displaystyle 1+c(c-1) &1+(c+1)(c-1)& 1+(c-1)^2
	\end{array}\right)\mbox{ as }n\rightarrow \infty.
\end{equation}

Since $\det( \widetilde{B})=\frac{6}{M^2}\neq 0$ and $\widetilde{B}$ is positive defined, it follows that there are two constants $0<\textfrak{a}_1'<\textfrak{a}_2'$ such that 
\begin{equation}\label{eq:specbn22}
	\operatorname{Spec}( \widetilde{B})\subset [\textfrak{a}_1',\textfrak{a}_2'].
\end{equation}
From \eqref{eq:inspei2} and \eqref{eq:specbn22}, we can deduce that the claim \eqref{eq:specbn2} holds.

{\bf Step 4}. Using \eqref{eq:ri1cv} and \eqref{eq:specbn2} we get that
 
\begin{align}\label{eq:insp2}
    \textfrak{a}_1\left(|a_n^1|^2+|a_n^2|^2+|a_n^3|^2\right)\leq
    \norm{B_n\left(\begin{array}{l}a_n^1\\a_n^2\\a_n^3\end{array}\right)}{2}^2\leq
    \textfrak{a}_2\left(|a_n^1|^2+|a_n^2|^2+|a_n^3|^2\right).
\end{align}
These inequalities together with \eqref{eq:ri1cv} imply that
\begin{align}\label{eq:inri2}
	2 \textfrak{a}_1 \sum_{n\in\mathbb{Z}^*,\,1\leq j\leq 3}|a_n^j|^2\leq \norm{\sum_{n\in\mathbb{Z}^*,\,1\leq j\leq 3}a_n^j\rho_{|n|}^\sigma\Psi_n^j\,}{X_{-\sigma}}^2 &\leq 2 \textfrak{a}_2 \sum_{n\in\mathbb{Z}^*,\,1\leq j\leq 3}|a_n^j|^2.
\end{align}
The estimates \eqref{eq:inri2} show that $\left(\rho_{|n|}^\sigma\Psi_n^j\right)_{n\in\mathbb{Z}^*,\,1\leq j\leq 3}$ is a Riesz basis in  $X_{-\sigma}$. The proof is finished.
\end{proof}

\subsection{Gap properties of the spectrum of the operator associated with the system \eqref{wave_mem_adj_syst_cv}}

Here we analyze the distance between the three families composing the spectrum of the operator associated to the system \eqref{wave_mem_adj_syst_cv}. As it is well-known, these gap properties shall play a fundamental role in the proof of our controllability result. 

%

Since there is no major change in the spectrum if $c$ is replaced by $-c$, we shall limit our analysis to the case $c>0$. Let us start by studying the distance between the elements of $(\lambda_n^1)_{n\in\mathbb{Z}^*}$ and those of  $(\lambda_n^2)_{n\in\mathbb{Z}^*}$ and $(\lambda_n^3)_{n\in\mathbb{Z}^*}$. 

\begin{lem}\label{lemma:dist1} 
Let $\frac 12<s<1$.
For any $n,m\in\mathbb{Z}^*$ and $k\in\{2,3\}$ we have that 
\begin{equation}\label{eq:diff123} 
\left|\lambda_n^1-\lambda_m^k\right|\geq \frac{3|M|}{\frac{2M^2}{\mu_1}+2}>0.
\end{equation}
Moreover,  for any $n\neq m$ we have that
\begin{equation}\label{eq:diff11}
    \left|\lambda_n^1-\lambda_m^1\right|\geq c\left|\mu_{|n|}^{\frac{1}{2s}}-\mu_{|m|}^{\frac{1}{2s}}\right|>0.
\end{equation}
\end{lem}

\begin{proof}
We remark that, if $k\in\{2,3\}$, the numbers  $\operatorname{Re}(\mu_n^1)$ and $\operatorname{Re}(\mu_m^k)$ have opposite signs. By taking into account \eqref{eq:red}, we deduce that
\begin{align*}
	\left|\lambda_n^1-\lambda_m^k\right|\geq\left|\operatorname{Re}\left(\lambda_n^1-\lambda_m^k\right)\right|=\left|\operatorname{Re}\left(\mu_{|n|}^1-\mu_{|m|}^k\right)\right|=\left|\mu_{|n|}^1+\frac{\mu_{|m|}^1}{2}\right|\geq \frac{3|M|}{\frac{2M^2}{\mu_1}+2},
\end{align*}    
and \eqref{eq:diff123} is proved. On the other hand, since $\mu_n^1\in\mathbb{R}$, for all $n\neq m$ we have that
\begin{align*}
    \left|\lambda_n^1-\lambda_m^1\right|\geq \left|\operatorname{Im}\left(\lambda_n^1-\lambda_m^1\right)\right|=c\left|\mu_{|n|}^{\frac{1}{2s}}-\mu_{|m|}^{\frac{1}{2s}}\right|> 0.
\end{align*}
The proof of the lemma is complete.
\end{proof}

Lemma \ref{lemma:dist1} shows that each element of the sequence $(\lambda_n^1)_{n\in\mathbb{Z}^*}$ is well separated from the others elements of the spectrum of $\mathcal{A}_c$. Let us now analyze the case of the other two families. We define the set

\begin{align}\label{spaceV}
	\mathcal V:=\left\{ \sqrt{3\left(\frac{\mu_n^1}{2\rho_n^{\frac{1}{2s}}}\right)^2 + \rho_n^{1-\frac{1}{s}}}\,:\,n\geq 1\right\},
\end{align}
and we will use the notation

\begin{align}\label{eq:noset}
	S=\{(n,j)\,:\,n\in\mathbb{Z}^*,\,\, 1\leq j\leq 3\}.
\end{align} 

\begin{lem}\label{lemma:dist20} 
Let $\frac 12 <s<1$, $\gamma=\gamma(s)\ge \frac{\pi}{2}$ the  constant given in \eqref{Gap} and let $c\in (0,\gamma)\cup (\gamma,\infty)$. Then the following assertions hold.
\begin{enumerate} 
	\item[(a)] If $c\notin{\mathcal V}$, then for each $n\geq 1$ there exists a constant $\upsilon=\upsilon(N,c)>0$ such that 
	
    \begin{equation}\label{eq:diff023} 
	    \left|\lambda_n^j-\lambda_m^k\right|\geq \upsilon(N,c),
	\end{equation}
    for any $(n,j),(m,k)\in S$ with $(n,j)\neq (m,k)$, $1\leq |n|,|m|\leq N$ and $ j,k\in\{2,3\}$.

    \item[(b)] If $c\in {\mathcal V}$, then there exists a unique $n_c\geq 1$ such that 
    \begin{equation}\label{eq:surprise}
	    \lambda_{-n_c}^2=\lambda_{n_c}^3,
	\end{equation} 
	and for each $n\geq 1$  there exists a constant $\upsilon=\upsilon(N,c)>0$ such that 
    \begin{equation}\label{eq:diff023n} 
	    \left|\lambda_n^j-\lambda_m^k\right|\geq \upsilon(N,c),
	\end{equation}
    for any $(n,j),(m,k)\in S$ with $(n,j)\neq (m,k)$, $1\leq |n|,|m|\leq N$ and $ j,k\in\{2,3\}$, with the exception of the cases $(n,j)= (- n_c,2)$ and $(m,k)=(n_c,3)$ given by \eqref{eq:surprise}.
\end{enumerate}
\end{lem}

\begin{proof} 
For any $(n,j),(m,k)\in S$, we have that  
\begin{align*}
    \left|\lambda_{n}^j-\lambda_{m}^k\right|>\left|\operatorname{Re}(\lambda_{n}^j)-\operatorname{Re}(\lambda_{m}^k)\right|=
    \frac 12\left|\mu_{|n|}^1-\mu_{|m|}^1\right|.
\end{align*}
Since Lemma \ref{lemma:sir} ensures that the sequence of real numbers $\left(|\mu_n^1|\right)_{n\geq 1}$ is increasing, it follows that
\begin{align*}
	\inf \Big\{\left|\lambda_{n}^j-\lambda_{m}^k\right|\,:\, (n,j),(m,k)\in S, \,\, |n|\neq |m|,\,\, 1\leq |n|,|m|\leq N,\,\, 2\leq j,k\leq 3\Big\}>0.
\end{align*}
It remains to study the case $|n|=|m|$. If $j\in\{2,3\}$ and $n\geq 1$, then 
\begin{align*}
    \left|\lambda_{n}^j-\lambda_{-n}^j\right|>\left|\operatorname{Im}(\lambda_{n}^j)-\operatorname{Im}(\lambda_{-n}^j)\right|=2c\rho_{n}^{\frac{1}{2s}}\geq 2c\rho_{1}^{\frac{1}{2s}}>0.
\end{align*}
Moreover, if $n\geq 1$, then
\begin{align*}
    &\left|\lambda_{n}^2-\lambda_{n}^3\right|>\left|\operatorname{Im}(\lambda_{n}^2)-\operatorname{Im}(\lambda_{n}^3)\right|=2\sqrt{3\left(\frac{\mu_n^1}{2}\right)^2+\rho_n}\geq 2\rho_1^{\frac 12}>0
    \\
    &\left|\lambda_{n}^2-\lambda_{-n}^3\right|>\left|\operatorname{Im}(\lambda_{n}^2)-\operatorname{Im}(\lambda_{-n}^3)\right|=2c\rho_n^{\frac{1}{2s}}+2\sqrt{3\left(\frac{\mu_n^1}{2}\right)^2+\rho_n}\geq 2\rho_1^{\frac 12}>0
    \\
    &\left|\lambda_{-n}^2-\lambda_{-n}^3\right|>\left|\operatorname{Im}(\lambda_{-n}^2)-\operatorname{Im}(\lambda_{-n}^3)\right|=2\sqrt{3\left(\frac{\mu_n^1}{2}\right)^2+\rho_n}\geq 2\rho_1^{\frac 12}>0.
\end{align*}
Finally, we remark that
\begin{align*}
    \left|\lambda_{-n}^2-\lambda_{n}^3\right|>\left|\operatorname{Im}(\lambda_{-n}^2)-\operatorname{Im}(\lambda_{n}^3)\right|=\left|-2c\rho_n^{\frac{1}{2s}}+2\sqrt{3\left(\frac{\mu_n^1}{2}\right)^2+\rho_n}\right|.
\end{align*}  
 
The last expression is zero if  and only if there exists $n_c>0$ such that $c=\sqrt{3\left(\frac{\mu_{n_c}^1}{2\rho_{n_c}^{\frac{1}{2s}}}\right)^2 + \rho_{n_c}^{1-\frac{1}{s}}}$, that is, if and only if  $c\in{\mathcal V}$ and the proof is complete.
\end{proof}

\begin{rem} 
{\em 
Lemma \ref{lemma:dist20} shows that, if $c\in{\mathcal V}$, then there exists a unique double  eigenvalue $\lambda_{-n_c}^2$. Its geometric multiplicity is two, since there are two  corresponding linearly independent eigenfunctions $\Psi^2_{- n_c}$ and $\Psi^3_{n_c}$.
}
\end{rem}

Now we analyze the higher part of the spectrum. The properties of the large elements of the sequences $(\lambda_n^2)_{n\in\mathbb{Z}^*}$ and $(\lambda_n^3)_{n\in\mathbb{Z}^*}$ are described in the following lemma.

\begin{lem} \label{lemma:inc} 
Let $\frac 12<s<1$ and $\gamma=\gamma(s)\ge\frac{\pi}{2}$ the constant given in \eqref{Gap}. Then for any $\epsilon>0$, there exists $N_\epsilon^1\in\mathbb{N}$ such that the following assertions hold.
\begin{enumerate}
    \item[(a)] If $c\in(0,\gamma)$, then we have the following.
    \begin{itemize}
        \item The sequences $\left(\operatorname{Im}(\lambda_n^2)\right)_{n\geq 1}$ and $\left(\operatorname{Im}(\lambda_{-n}^2)\right)_{n\geq N_\epsilon^1}$ are increasing and included in the intervals $\left[(1+c)\rho_1^{\frac 12},\infty\right)$ and $\left[\left(1-\frac c\gamma\right)\rho_{1}^{\frac{1}{2}},\infty\right)$, respectively.

        \item The sequences $\left(\operatorname{Im}(\lambda_n^3)\right)_{n\geq N_\epsilon^1}$ and $\left(\operatorname{Im}(\lambda_{-n}^3)\right)_{n\geq 1}$ are decreasing and included in the intervals $\left(-\infty,\left(-1+\frac c\gamma\right)\rho_{1}^{\frac{1}{2}}\gamma\right]$ and $\left(-\infty,-(1+c)\rho_1^{\frac 12}\right]$, respectively.
    \end{itemize}
    Moreover, 
    \begin{equation}\label{eq:dinm1}
	    \begin{array}{ll}
		    \operatorname{Im}(\lambda_{n+1}^2)-\operatorname{Im}(\lambda_{n}^2)=\operatorname{Im}(\lambda_{-n}^3)-\operatorname{Im}(\lambda_{-n-1}^3)\geq c\gamma-\epsilon>0 & \mbox{ for }n\geq N_\epsilon^1, 
	        \\[8pt]
	        \operatorname{Im}(\lambda_{-n-1}^2)-\operatorname{Im}(\lambda_{-n}^2)=\operatorname{Im}(\lambda_{n}^3)-\operatorname{Im}(\lambda_{n+1}^3)\geq \left(1-\frac c\gamma\right)\left(\rho_{n+1}^{\frac 12}-\rho_n^{\frac 12}\right)-\epsilon>0 & \mbox{ for }n\geq N_\epsilon^1,  
	        \\[8pt]
	       \operatorname{Im}(\lambda_{-n}^2)\leq \operatorname{Im}(\lambda_{-N_\epsilon^1}^2),\quad\operatorname{Im}(\lambda_{n}^3)\geq \operatorname{Im}(\lambda_{N_\epsilon^1}^3)  & \mbox{ for }1\leq n\leq N_\epsilon^1.
	        \\[10pt]
        \end{array}
    \end{equation}

    \item[(b)] If $c\in(\gamma,\infty)$, then we have the following.
    \begin{itemize}
        \item The sequences $\left(\operatorname{Im}(\lambda_n^2)\right)_{n\geq 1}$ and $\left(\operatorname{Im}(\lambda_{n}^3)\right)_{n\geq N_\epsilon^1}$ are increasing and included in the intervals $\left[(c+1)\rho_1^{\frac 12},\infty\right)$ and $\left[\left(\frac c\gamma-1\right)\rho_{1}^{\frac{1}{2}},\infty\right)$, respectively.

        \item The sequences $\left(\operatorname{Im}(\lambda_{-n}^2)\right)_{n\geq N_\epsilon^1}$ and $\left(\operatorname{Im}(\lambda_{-n}^3)\right)_{n\geq 1}$ are decreasing and included in the intervals $\left(-\infty,\left(1-\frac c\gamma\right)\rho_{1}^{\frac{1}{2}}\right]$ and $\left(-\infty,-(c+1)\rho_1^{\frac 12}\right]$, respectively.
    \end{itemize}
    Moreover, 
    \begin{equation}\label{eq:dinm2}
        \begin{array}{ll}
	        \operatorname{Im}(\lambda_{n+1}^2)-\operatorname{Im}(\lambda_{n}^2)=\operatorname{Im}(\lambda_{-n}^3)-\operatorname{Im}(\lambda_{-n-1}^3)\geq c\gamma-\epsilon & \mbox{ for }n\geq N_\epsilon^1, 
	        \\[8pt]
	        \operatorname{Im}(\lambda_{-n}^2)-\operatorname{Im}(\lambda_{-n-1}^2)=\operatorname{Im}(\lambda_{n}^3)-\operatorname{Im}(\lambda_{n+1}^3)\geq \left(\frac c\gamma-1\right)\left(\rho_{n+1}^{\frac 12}-\rho_n^{\frac 12}\right)-\epsilon>0 & \mbox{ for }n\geq N_\epsilon^1,
	        \\[8pt]
	      \operatorname{Im}(\lambda_{n}^3)\leq \operatorname{Im}(\lambda_{N_\epsilon^1}^3),\quad \operatorname{Im}(\lambda_{-n}^2)\geq \operatorname{Im}(\lambda_{-N_\epsilon^1}^2)  & \mbox{ for }1\leq n\leq N_\epsilon^1.
        \end{array}
    \end{equation}
    \end{enumerate}
\end{lem}

\begin{proof}
We first notice that
\begin{equation}\label{eq:leg23}
    \overline{\lambda}_n^2=-i\,\sgn(n)c\rho_{|n|}^{\frac{1}{2s}}-\frac{\mu^1_{|n|}}{2}-i\sqrt{3\left(\frac{\mu^1_{|n|}}{2}\right)^2+n^2}=\lambda_{-n}^3,\qquad n\in\mathbb{Z}^*.
\end{equation}

Hence, we can deduce that $\operatorname{Im}(\lambda^3_{n})=-\operatorname{Im}(\lambda^2_{-n})$.  Since the sequence $\left(\mu_{n}^1\right)_{n\geq 1}$ is bounded (by \eqref{eq:red}), given any $\epsilon>0$ there exists $N_\epsilon^1\in\mathbb{N}$ such that 
\begin{align}\label{eq:neps}
    \frac{\frac{3}{4}\left(\mu^1_{n}\right)^2}{\sqrt{\frac{3}{4}\left(\mu^1_{n}\right)^2+\rho_n}+\rho_n^{\frac 12}}\leq \epsilon,\qquad n\geq N_\epsilon^1.
\end{align}
We analyze the cases $c\in(0,\gamma)$ only, the other one being similar. 
\begin{itemize}
    \item For the sequence $\left(\operatorname{Im}(\lambda_n^2)\right)_{n\geq 1}$, since $\frac 12<s<1$ we have that
    \begin{align*}
        \operatorname{Im}(\lambda_{n}^2)= c\rho_{n}^{\frac{1}{2s}}+\sqrt{3\left(\frac{\mu^1_{|n|}}{2}\right)^2+\rho_n}\geq c\rho_{n}^{\frac{1}{2s}}+\rho_{n}^{\frac{1}{2}}\geq (1+c)\rho_{n}^{\frac{1}{2}}\geq (1+c)\rho_{1}^{\frac{1}{2}}, 
    \end{align*}    
    and, according to \eqref{eq:neps}, for any $n\geq N_\epsilon^1$, it follows that
    \begin{align*}
       \operatorname{Im}&(\lambda_{n+1}^2)-\operatorname{Im}(\lambda_{n}^2) 
        \\
        &=c\left(\rho_{n+1}^{\frac{1}{2s}}-\rho_{n}^{\frac{1}{2s}}\right) +\frac{\frac{3}{4}\left(\mu^1_{n+1}\right)^2}{\sqrt{\frac{3}{4}\left(\mu^1_{n+1}\right)^2+\rho_{n+1}}+\rho_{n+1}^{\frac 12}}-\frac{\frac{3}{4}\left(\mu^1_{n}\right)^2}{\sqrt{\frac{3}{4}\left(\mu^1_{n}\right)^2+\rho_n}+\rho_n^{\frac 12}}+\left(\rho_{n+1}^{\frac 12}-\rho_n^{\frac 12}\right)
	    \\
	    &\geq c\gamma-\epsilon>0,
	\end{align*}
	for some constant $\gamma\geq\frac \pi2$ (see \eqref{Gap}).            
	\item  As for the sequence $\left(\operatorname{Im}(\lambda_{-n}^2)\right)_{n\geq 1}$ we remark that
	\begin{align*}
	    \operatorname{Im}(\lambda_{-n}^2)= -c\rho_{n}^{\frac{1}{2s}}+\sqrt{3\left(\frac{\mu^1_{|n|}}{2}\right)^2+\rho_n}\geq -c\rho_{n}^{\frac{1}{2s}} + \rho_{n}^{\frac{1}{2}}\geq \left(1-\frac c\gamma\right)\rho_{n}^{\frac{1}{2}} \geq \left(1-\frac c\gamma\right)\rho_{1}^{\frac{1}{2}},
	\end{align*}     
	and, similarly as above, for any $n\geq N_\epsilon^1$, we have that
	\begin{align*}
	\operatorname{Im}&(\lambda_{-n-1}^2)-\operatorname{Im}(\lambda_{-n}^2) 
	\\
	&=-c\left(\rho_{n+1}^{\frac{1}{2s}}-\rho_{n}^{\frac{1}{2s}}\right) +\frac{\frac{3}{4}\left(\mu^1_{n+1}\right)^2}{\sqrt{\frac{3}{4}\left(\mu^1_{n+1}\right)^2+\rho_{n+1}}+\sqrt{\rho_{n+1}}}-\frac{\frac{3}{4}\left(\mu^1_{n}\right)^2}{\sqrt{\frac{3}{4}\left(\mu^1_{n}\right)^2+\rho_n}+\sqrt{\rho_n}}+\left(\rho_{n+1}^{\frac 12}-\rho_n^{\frac 12}\right)
	\\
	&\geq \left(1-\frac c\gamma\right)\left(\rho_{n+1}^{\frac 12}-\rho_n^{\frac 12}\right)-\epsilon>0.
	\end{align*}
\end{itemize}
	
Hence, all the desired properties of the sequence  $\left(\operatorname{Im}(\lambda_{n}^2)\right)_{n\in\mathbb{Z}^*}$ are proved in the case $c\in (0,\gamma)$. Since $\operatorname{Im}(\lambda_n^3)=-\operatorname{Im}(\lambda_{-n}^2)$, the properties of the sequence  $\left(\operatorname{Im}(\lambda_{n}^3)\right)_{n\in\mathbb{Z}^*}$  follow immediately from those of  $\left(\operatorname{Im}(\lambda_{n}^2)\right)_{n\in\mathbb{Z}^*}$. The proof is finished.
\end{proof}

Now we  study the possible interactions between the large elements of the sequences $\left(\lambda_{n}^2\right)_{n\in\mathbb{Z}^*}$  and $\left(\lambda_{n}^3\right)_{n\in\mathbb{Z}^*}$.

\begin{lem}\label{lemma.dist23} 
Let $\epsilon>0$ sufficiently small, $\frac 12<s<1$ and $\gamma=\gamma(s)\ge \frac{\pi}{2}$ the constant given in \eqref{Gap}. Then there exist $N_\epsilon\geq 1$ and two positive constants $\delta(c,\epsilon)$ and $\delta'(c,\epsilon)$ with the property that for each $m\geq N_\epsilon$ there exists $n_m\geq N_\epsilon$ such that the following assertions hold.
\begin{enumerate}
    \item[(a)] If $c\in (0,\gamma)$, then
    \begin{equation}\label{eq:di0}
        \frac{1-c}{2}+3\epsilon \geq  \left|\lambda_{m}^2-\lambda_{-n_m}^2\right|=\left|\lambda_{-m}^3-\lambda_{n_m}^3\right|\geq \frac{\delta'(c,\epsilon)}{\rho_m},
    \end{equation}
    and
    \begin{align}\label{eq:di1}
        \inf\left\{\left|\lambda_{m}^2-\lambda_{n}^2\right|\,:\,|n|\geq N_\epsilon,\,\, n\neq -n_m \right\}     =\inf\left\{\left|\lambda_{-m}^3-\lambda_{n}^3\right|\,:\,|n|\geq N_\epsilon,\,\, n\neq n_m \right\}\geq \delta(c,\epsilon).
    \end{align}

    \item[(b)] If $c\in (\gamma,\infty)$, then
    \begin{equation}\label{eq:di3}
        \frac{c-1}{2}+3\epsilon\geq  \left|\lambda_{m}^2-\lambda_{n_m}^3\right|=\left|\lambda_{-m}^3-\lambda_{-n_m}^2\right|\geq \frac{\delta'(c,\epsilon)}{\rho_m},
    \end{equation}
    and
    \begin{align}\label{eq:di4}
        \inf\left\{\left|\lambda_{m}^2-\lambda_{n}^3\right|\,:\,|n|\geq N_\epsilon,\,\, n\neq n_m \right\}        =\inf\left\{\left|\lambda_{-m}^3-\lambda_{n}^2\right|\,:\,|n|\geq N_\epsilon,\,\, n\neq- n_m \right\}\geq \delta(c,\epsilon).
    \end{align}
\end{enumerate}
\end{lem}

\begin{proof} 
Firstly, notice that by  \eqref{eq:eig} there exists $N^2_\epsilon$ such that 
\begin{equation}\label{eq:remic}
	\left|\operatorname{Re}(\lambda_n^j)+\frac{M}{2}\right|\leq \frac{\epsilon}{2},\qquad |n|\geq N^2_\epsilon,\,\, j\in\{2,3\}.
\end{equation}
Let $m \geq N_\epsilon:=\max\{N^1_\epsilon,N^2_\epsilon\}$, where $N^1_\epsilon$ is the number given by Lemma \ref{lemma:inc}. Let $n_m\in\mathbb{N}$ be such that
\begin{equation}\label{eq:nm}
	\left|\Big|1-\frac c\gamma\Big|\rho_{n_m}^{\frac 12}-(1+c)\rho_{m}^{\frac 12}\right| =\inf_{n\geq 1}\left|\Big|1-\frac c\gamma\Big|\rho_{n}^{\frac 12}-(1+c)\rho_{m}^{\frac 12}\right|.
\end{equation}
Notice that $n_m$ also verifies
\begin{equation}\label{eq:nmint}
   -\frac{1}{2}+\frac{\gamma(1+c)}{|\gamma-c|}\rho_m^{\frac 12}\leq \rho_{n_m}^{\frac 12}\leq \frac{1}{2}+\frac{\gamma(1+c)}{|\gamma-c|}\rho_m^{\frac 12}.
\end{equation}
We analyze separately the following two cases.

\paragraph{\bf Case 1} Let $c\in(0,\gamma)$. For each $m\geq N_\epsilon$ we have that
\begin{equation}\label{eq:innm}
	\left|\lambda_{m}^2-\lambda_{-n_m}^2\right|=\inf_{n\in \mathbb{Z}^*}\left|\lambda_{m}^2-\lambda_{n}^2\right|,
\end{equation}
with $n_m$ given by \eqref{eq:nm}. Indeed, we have that
\begin{align*}
	\left|\lambda_{m}^2-\lambda_{-n_m}^2\right| \leq &\, \left|\operatorname{Im}(\lambda_{m}^2-\lambda_{-n_m}^2)\right|+\left|\operatorname{Re}(\lambda_{m}^2-\lambda_{-n_m}^2)\right|
	\\
	\leq &\, \left|\left(1-\frac c\gamma\right)\rho_{n_m}^{\frac 12}-(1+c)\rho_{m}^{\frac 12}\right|+\left|(1+c)\rho_{m}^{\frac 12}-\operatorname{Im}(\lambda_{m}^2)\right|	+\left|\operatorname{Im}(\lambda_{-n_m}^2)-\left(1-\frac c\gamma\right)\rho_{n_m}^{\frac 12}\right|
	\\
	&+\left|\operatorname{Re}(\lambda_{m}^2-\lambda_{-n_m}^2)\right|,
\end{align*}
from which, by taking into account \eqref{eq:neps} and \eqref{eq:remic}, we deduce that
\begin{equation}\label{eq:in1nm}
    \left|\lambda_{m}^2-\lambda_{-n_m}^2\right|\leq  \left|\left(1-\frac c\gamma\right)\rho_{n_m}^{\frac 12}-(1+c)\rho_{m}^{\frac 12}\right|+3\epsilon.
\end{equation}
On the other hand, from \eqref{eq:dinm1} and \eqref{eq:neps} we can deduce that
\begin{align*}
    |&\lambda_{m}^2-\lambda_{n}^2| 
    \\
    &\geq \min\left\{\left|\lambda_{m}^2-\lambda_{|n|}^2\right|,\,\,\left|\lambda_{m}^2-\lambda_{-|n|}^2\right| \right\} \geq \min\left\{c\gamma-\epsilon,\,\, \left|\operatorname{Im}(\lambda_{m}^2-\lambda_{-|n|}^2)\right| \right\}
    \\
    &\geq \min\left\{c\gamma-\epsilon,\,\, \left|\left(1-\frac c\gamma\right)\rho_{|n|}^{\frac 12}-(1+c)\rho_{m}^{\frac 12}\right| -\left|(1+c)\rho_{m}^{\frac 12}-\operatorname{Im}(\lambda_{m}^2)\right|-\left|\operatorname{Im}(\lambda_{-|n|}^2)+\left(1-\frac c\gamma\right)\rho_{|n|}^{\frac 12}\right|\right\}
    \\
    &\geq \min\left\{c\gamma-\epsilon,\,\, \left|\left(1-\frac c\gamma\right)\rho_{|n|}^{\frac 12}-(1+c)\rho_{m}^{\frac 12}\right| -2\epsilon\right\},
\end{align*}
which, by taking into account \eqref{eq:nmint}, implies that
\begin{align}\label{eq:in2nm}
    |\lambda_{n}^2-\lambda_{m}^2| \geq \min\left\{\left|\left(1-\frac c\gamma\right)\rho_{n_m}^{\frac 12}-(1+c)\rho_{m}^{\frac 12}\right|+2-\epsilon,\,\, \frac{\gamma-c}{2}+ \left|\left(1-\frac c\gamma\right)\rho_{n_m}^{\frac 12}-(1+c)\rho_{m}^{\frac 12}\right|-2\epsilon\right\}.\notag
    \\
\end{align}
 
From \eqref{eq:in1nm}-\eqref{eq:in2nm} we can deduce that, for $\epsilon$ sufficiently small, \eqref{eq:innm} holds true. It follows that, for each $m\geq N_\epsilon$, we have 
\begin{equation}
    \inf_{n\in\mathbb{Z}^*,\,\, n\neq n_m}\left|\lambda_{m}^2-\lambda_{n}^2\right|\geq \delta(\epsilon,c):=\min\left\{2-\epsilon,\,\, \frac{\gamma-c}{2}-2\epsilon\right\},
\end{equation}
and
\begin{equation}
    \left|\lambda_{m}^2-\lambda_{-n_m}^2\right|\geq \left|\operatorname{Re}(\lambda_{m}^2)-\operatorname{Re}(\lambda_{-n_m}^2)\right|\geq \frac{\delta'(c,\epsilon)}{\rho_m}.
\end{equation}

\paragraph{\bf Case 2}
Let $c\in(\gamma,\infty)$. As before, for $\epsilon$ small enough and for each $m\geq N_\epsilon$, we have that
\begin{equation}
    \left|\lambda_{m}^2-\lambda_{n_m}^3\right|=\inf_{n\in\mathbb{Z}^*}\left|\lambda_{m}^2-\lambda_{n}^3\right|,
\end{equation}
with $n_m$ given by \eqref{eq:nm} and consequently
\begin{equation}
    \inf_{n\in\mathbb{Z}^*}\left|\lambda_{m}^2-\lambda_{n}^3\right|\geq \left|\operatorname{Re}(\lambda_{m}^2)-\operatorname{Re}(\lambda_{n_m}^3)\right|\geq \frac{\delta'(c,\epsilon)}{\rho_m}.
\end{equation}
The rest of the proof is similar to the case $c\in(0,\gamma)$.
\end{proof}

\begin{rem} 
{\em
Lemmas \ref{lemma:dist1}-\ref{lemma.dist23} show that all elements of the spectrum $\sigma(\mathcal A_c)=
\left(\lambda_n^j\right)_{(n,j)\in S}$ are well separated one from another except for the following special cases.

\begin{enumerate}

\item[(a)] If $c\in (0,\gamma)$,  then the eigenvalues $\lambda^2_m$ and $\lambda^2_{-n_m}$ have a distance at least of order $\frac{1}{\rho_m}$ between them and a similar relation holds for $\lambda^3_{m}$ and $\lambda^3_{-n_m}$.

\item[(b)] If $c\in(\gamma,\infty)$, then the eigenvalues $\lambda^2_m$ and $\lambda^3_{n_m}$ have a distance at least of order $\frac{1}{\rho_m}$ between them and a similar relation holds for $\lambda^3_{-m}$ and $\lambda^2_{-n_m}$.

\item[(c)] If $c\in{\mathcal V}$, then there exists a unique double eigenvalues  $\lambda_{-n_c}^2=\lambda_{n_c}^3$.

\end{enumerate}

Even though the asymptotic gap between the elements of the spectrum is equal to zero, the fact that we know the velocities with which the distances between these eigenvalues tend to zero, this will allow us to estimate the norm of the biorthogonal sequence associated to $\left(e^{-\lambda_n^j t}\right)_{(n,j)\in S}$ that we will study in the next section.
}
\end{rem}

\section{Construction of a biorthogonal sequence}\label{bio_sec}

In this section we construct and evaluate a biorthogonal sequence $\left(\theta_m^k\right)_{(m,k)\in S}$ in $L^2\left(-\frac{T}{2},\frac{T}{2}\right)$ associated  to the family of exponential functions $\Lambda=\left(e^{-\lambda_{n}^jt}\right)_{(n,j)\in S},$ where $\lambda_n^j$ are given by \eqref{eq:eigneq}. In order to avoid
the double eigenvalue, which according to Lemma \ref{lemma:dist20} occurs if $c\in\mathcal{V}$, and to keep the notation as simple as possible, we make the convention that, if $c\in\mathcal{V}$, we redefine $ \lambda_{-n_c}^2$ as follows
\begin{align}\label{eq:condob}
	\lambda_{-n_c}^2=-i\,\sgn(n_c)c\rho_{|n_c|}^{\frac{1}{2s}}+i\sqrt{3\left(\frac{\mu_{|n_c|}^1}{2}\right)^2+\rho_{|n_c|}}-\frac{1}{2}i+\frac{\mu_{|n_c|}^1}{2}.
\end{align}

In this way Lemmas \ref{lemma:dist1},  \ref{lemma:dist20} and \ref{lemma.dist23} guarantee that all elements of the family $\left(\lambda_n^j\right)_{(n,j)\in S}$ are different. Since the biorthogonal sequence has the property that
\begin{align*}
	\int_{-\frac{T}{2}}^{\frac{T}{2}}\theta_m^k(t)e^{-\overline{\lambda}_{n}^jt}\,{dt}=\delta_{mk}^{nj},	
\end{align*}
if we define the Fourier transform of $\theta_m^k$,
\begin{align*}
	\widehat{\,\theta}_m^k(z)=\int_{-\frac{T}{2}}^{\frac{T}{2}}\theta_m^k(t)e^{-izt}\,{dt},
\end{align*}
we obtain that
\begin{align}\label{eq:cotr}
	\widehat{\,\theta}_m^k(-i\overline{\lambda}_n^j)=\delta_{mk}^{nj},\qquad (n,j),(m,k)\in S.
\end{align} 
Therefore, we define the infinite product
\begin{align}\label{eq:prod}
	P(z)=z^3\prod_{(n,j)\in S}\left(1+\frac{z}{i\overline{\lambda}_n^j}\right):=z^3\lim_{R\rightarrow\infty}\prod_{\substack{(n,j)\in S\\|\lambda_n^j|\leq R}}\left(1+\frac{z}{i\overline{\lambda}_n^j}\right),
\end{align}
and we study some of its properties in the next theorem. 
We shall prove that the limit in \eqref{eq:prod} exists and defines an entire function. This and other important properties of $P(z)$ are given in the following theorem.

\begin{thm}\label{te:pprod}
Let $\frac 12 <s<1$, $\gamma=\gamma(s)\ge \frac{\pi}{2}$ the constant given in \eqref{Gap}, $c\in\mathbb{R}\setminus\{-\gamma,0,\gamma\}$ and let $P$ be given by \eqref{eq:prod}. Then the following assertions hold.
\begin{enumerate}
	\item[(a)] $P$ is well defined, and it is an entire function of exponential type $\left(\frac{1}{|c|}+\frac{1}{|c+\gamma|}+\frac{1}{|c-\gamma|}\right)\pi$.

	\item[(b)] For each $\delta>0$, there exists a constant $C(\delta)>0$ such that 
	\begin{equation}\label{eq:preal}
		\left|P(z)\right|\leq C(\delta),\qquad z=x+iy,\,\,x,y\in\mathbb{R},\,\,|y|\leq \delta.
	\end{equation}  
	Moreover, there exists a constant $C_1>0$ such that, for any $(m,k)\in S$, 
	\begin{equation}\label{eq:preal2}
		\left|\frac{P(x)}{x+i\overline{\lambda}_m^k}\right| \leq \frac{C_1}{1+\left|x+\operatorname{Im}(\lambda_m^k)\right|},\qquad z=x+iy,\,\,x,y\in\mathbb{R},\,\,|y|\leq \delta.
	\end{equation}

	\item[(c)] Each point $-i\overline{\lambda}_m^j$ is a simple zero of $P$ and there exists a constant $C_2>0$ such that
	\begin{equation}\label{eq:der}
		\left|P'(-i\overline{\lambda}_m^j)\right|\geq \frac{C_2}{\rho_m},\qquad (m,j)\in S.
	\end{equation}
\end{enumerate}
\end{thm}

\begin{proof} 
The proof is based on the spectral analysis presented in Section \ref{spectrum_sect}, and it is totally analogous to the one in \cite[Theorem 5.1]{biccari2018null}. We thus omit it for brevity.
\end{proof}

\noindent The previous theorem allows us to construct the biorthogonal sequence we were looking for.

\begin{thm}\label{te:bio} 
Let $\frac 12 <s<1$, $\gamma=\gamma(s)\ge \frac{\pi}{2}$ the constant given in \eqref{Gap}, $c\in \mathbb{R}\setminus\{-\gamma,0,\gamma\}$ and let the constant $T>2\pi\left( \frac{1}{|c|}+\frac{1}{|c+\gamma|}+\frac{1}{|c-\gamma|}\right)$. Then there exist a biorthogonal sequence $\left(\theta_m^k\right)_{(n,j)\in S}$ associated to the family of complex exponentials $\left(e^{-\lambda_n^j t}\right)_{(n,j)\in S}$ in $L^2\left(-\frac{T}{2},\frac{T}{2}\right)$ and a  constant $C>0$ such that
\begin{equation}\label{eq:biono}
	\left\|\sum_{(m,k)\in S}\beta_m^k\theta_m^k\right\|^2_{L^2\left(-\frac{T}{2},\frac{T}{2}\right)}\leq C \sum_{(m,k)\in S}\rho_{|m|}^2\left|\beta_m^k\right|^2,
\end{equation}
for any finite sequence of complex numbers $\left(\beta_m^k\right)_{(m,k)\in S}$.
\end{thm}

\begin{proof}
We define the entire function
\begin{equation}\label{eq:wht}
	\widehat{\theta}_m^j(z)=\displaystyle \frac{P(z)}{P'(-i\overline{\lambda}_m^j)},
\end{equation} 
and let
\begin{equation}\label{eq:bt}
	\theta_m^j=\displaystyle \frac{1}{2\pi}\int_\mathbb{R} \widehat{\theta}_m^j(x)e^{ixt}\,dx.
\end{equation} 

It follows from Theorem \ref{te:pprod}  that $(\theta_m^j)_{(m,j)\in S}$ is a biorthogonal sequence associated to the family of exponential functions $\Lambda=\left(e^{-\lambda_{n}^jt}\right)_{(n,j)\in S}$ in $L^2\left(-\frac{T'}{2},\frac{T'}{2}\right)$, where  $T'= 2\pi\left( \frac{1}{|c|}+\frac{1}{|c+\gamma|}+\frac{1}{|c-\gamma|}\right)$. Moreover, we have that
\begin{align*}
	\|\theta_m^j\|_{L^2\left(-\frac{T'}{2},\frac{T'}{2}\right)}\leq C\, \rho_{|m|}, \qquad (m,j)\in S.
\end{align*}

An argument similar to \cite{kahane1962pseudo} (see also \cite[Proposition 8.3.9]{tucsnak2009observation}) allows us to prove that, for any $T>T'$, there exists a biorthogonal sequence  $(\theta_m^j)_{(m,j)\in S}$ associated to   $\Lambda=\left(e^{-\lambda_{n}^jt}\right)_{(n,j)\in S}$ in
$L^2\left(-\frac{T}{2},\frac{T}{2}\right)$ such that  \eqref{eq:biono} is verified. The proof is finished.
\end{proof}

The following immediate consequence of Theorem \ref{te:bio} will be very useful for the controllability problem that we shall study in the next section. 

\begin{cor}
For any finite sequence of scalars $(a_n^j)_{(n,j)\in S}\subset \mathbb{C}$, we have that

\begin{align}\label{eq:insum}
	\sum_{(n,j)\in S}\frac{\left|a_n^j\right|^2 }{\rho_{|n|}^2}\leq  C \norm{\sum_{(n,j)\in S}a_{n}^j e^{-\lambda_n^j t}}{L^2\left(-\frac{T}{2},\frac{T}{2}\right)}^2.
\end{align}
\end{cor}

\begin{proof} 
We have that
\begin{align*}
	\sum_{(n,j)\in S}\frac{\left|a_n^j\right|^2}{\rho_{|n|}^2} &= \int_{-\frac{T}{2}}^{\frac{T}{2}}
	\overline{\left( \sum_{(n,j)\in S} a_n^j e^{-\lambda_n^j t}\right)}\left( \sum_{(m,k)\in S} \frac{a_m^k }{\rho_{|m|}^2}\theta_m^k(t)\right)\,{ dt} 
	\\
	&\leq \norm{\sum_{(n,j)\in S}a_{n}^j e^{-\lambda_n^j t}}{L^2\left(-\frac{T}{2},\frac{T}{2}\right)}\, \norm{\sum_{(m,k)\in S}\frac{a_{m}^k}{\rho_{|m|}^2}\theta_m^k}{L^2\left(-\frac{T}{2},\frac{T}{2}\right)},
\end{align*}
from which, by taking into account \eqref{eq:biono}, we easily deduce \eqref{eq:insum}. 
\end{proof}

\section{Proof of the main result}\label{control_sect}

In this section we study the controllability properties of Equation \eqref{wave_mem} by proving the main result of the paper.
 To do this we need further preparation. Firstly, we shall reduce our original problem to a moment problem. Secondly, we will solve the moment problem with the help of the biorthogonal sequence that we have constructed in Section \ref{bio_sec}. Let us begin with the following result concerning the solutions of \eqref{wave_mem_adj_syst_cv}.

\begin{lem}\label{sol_adj_lemma}
For each initial data 
\begin{align}\label{in_data_adj}
	\left(\begin{array}{c}
		\varphi(T,x)\\ \varphi_t(T,x)\\ \psi(T,x)
	\end{array}\right) = \left(\begin{array}{c}
	\varphi^0(x)\\ \varphi^1(x)\\ \psi^0(x)
	\end{array}\right) = \sum_{(n,j)\in S} b_n^{\,j}\Psi_n^j(x) \in L^2(\I)\times\mathbb H_s^{-1}(\I)\times \mathbb H_s^{-1}(\I), 
\end{align} 
there exists a unique solution of Equation \eqref{wave_mem_adj_syst_cv} given by  	
\begin{align}\label{sol_adj}
	\left(\begin{array}{c} \varphi(t,x)\\ \varphi_t(t,x)\\ \psi(t,x)\end{array}\right) = \sum_{(n,j)\in S} b_n^{\,j}e^{\lambda_n^j(T-t)}\Psi_n^j(x).
\end{align}
\end{lem}

\begin{proof}
Since  $(\Psi_n^j)_{(n,j)\in S}$ is a Riesz basis (by Theorem \ref{te:lari}), it follows that each initial data in $ L^2(\I)\times\mathbb H_s^{-1}(\I)\times \mathbb H_s^{-1}(\I)$ can be written in the form \eqref{in_data_adj} with $(b_n^{\,j})_{(n,j)\in S}\in\ell^2$. We remark that, for $(n,j)\in S$, if we take as initial data $\Psi_n^j$, then the solution to \eqref{wave_mem_adj_syst_cv} is given by 
\begin{align}\label{sol_adj_expr}
	\left(\begin{array}{c} \varphi(t,x)\\ \varphi_t(t,x)\\ \psi(t,x)\end{array}\right) = e^{\lambda_n^j(T-t)}\Psi_n^j(x).
\end{align}
The proof is finished.
\end{proof}

We have the following result which reduces the controllability problem to a problem of moments.

\begin{lem}\label{moment_lemma}
Let $\sigma\ge 0$. Then Equation \eqref{wave_mem_syst_cv} is memory-type null controllable at time $T$ if and only if for each $(y^0,y^1)\in\mathbb H_s^{\sg+1}(\I)\times\mathbb H_s^{\sg}(\I)$,  
\begin{align}\label{data_dec}
	y^0(x) = \sum_{n\in\ZZ^*} y_0^ne^{i\,\sgn(n)\rho_{|n|}^{\frac{1}{2s}}x},\quad y^1(x) = \sum_{n\in\ZZ^*} y_1^ne^{i\,\sgn(n)\rho_{|n|}^{\frac{1}{2s}}x},
\end{align}
there exists $\widetilde{u}\in L^2(Q)$ such that 
\begin{align}\label{moment_pb}
	\int_0^T\int_{\omega_0} \tilde{u}(t,x)e^{-i\,\sgn(n)\rho_{|n|}^{\frac{1}{2s}}x}e^{-\bar{\lambda}_n^jt}\,dxdt = -2\Big(\bar{\mu}_{|n|}^jy_n^0 + y_n^1\Big),\qquad (n,j)\in S.
\end{align} 
\end{lem}

\begin{proof} 
Recall that, by Lemma \ref{control_id_lemma_syst}, Equation \eqref{wave_mem_syst_cv} is memory-type null controllable in time $T$ if and only if it holds the identity

\begin{align}\label{control_id_cv_2}
	\int_0^T\int_{\omega_0} \tilde{u}(t,x)\bar{\varphi}(t,x)\,dxdt = \big\langle y^0(\cdot),\varphi_t(0,\cdot)\big\rangle_{\mathbb H_s^{1}(\I),\mathbb H_s^{-1}(\I)} - \int_{\I} (y^1(x)+cy^0_x(x))\bar{\varphi}(0,x)\,dx,
\end{align}
where, for $(p^0,p^1,q^0)\in L^2(\I)\times\mathbb H_s^{-1}\times\mathbb H_s^1(\I)$, $(\varphi,\psi)$ is the unique solution to \eqref{wave_mem_adj_syst_cv}. In fact, it is sufficient to consider as initial data the elements of the Riesz basis $\left\{\Psi_n^j\right\}_{(n,j)\in S}:$
\begin{align*}
	\left(\begin{array}{c} p^0\\[5pt]p^1-cp_x^0\\[5pt]q^0\end{array}\right) = \Psi_n^j = \left(\begin{array}{c} 1\\[5pt]-\lambda_n^j\\[5pt] \displaystyle \frac{1}{\lambda_n^j-i\,\sgn(n)c\rho_{|n|}^{\frac{1}{2s}}}\end{array}\right)e^{i\,\sgn(n)\rho_{|n|}^{\frac{1}{2s}}x}.
\end{align*}

Then, by Lemma \ref{sol_adj_lemma}, the solution to \eqref{wave_mem_adj_syst_cv} can be written in the form \eqref{sol_adj_expr}. Moreover, we can readily check that 
\begin{align*}
    \int_{\I}\big(y^1+cy^0_x\big)\bar{\varphi}(0,x)\,dx = 2\left(y_n^1+i\,\sgn(n)c\rho_{|n|}^{\frac{1}{2s}} y_n^0\right)e^{\bar{\lambda}_n^jT},
\end{align*}
where we have used the orthogonality of the eigenfunctions in $L^2(\I)$. In a similar way, we also have
\begin{align*}
    \big\langle y^0,\varphi_t(0,\cdot)\big\rangle_{\mathbb H_s^{1}(\I),\mathbb H_s^{-1}(\I)} = -\sum_{(n,j)\in S}y_n^0\bar{\lambda}_n^je^{\bar{\lambda}_n^jT},
\end{align*}
and from \eqref{control_id_cv_2} we finally obtain \eqref{moment_pb}. The proof is finished.
\end{proof} 

In order to solve the moment problem \eqref{moment_pb}, we shall use the following result (see \cite[Chapter 4, Section 1, Theorem 2]{young2001introduction}).

\begin{thm}\label{thm_young}
Let $(f_n)_n$ be a sequence belonging to a Hilbert space $H$ and $(c_n)_n$ a sequence of scalars. In order that the equations
\begin{align*}
	(f,f_n)_H=c_n
\end{align*}
admit at least one solution $f\in H$ satisfying $\norm{f}{H}\leq M$, it is necessary and sufficient that 
\begin{align}\label{rel_young}
	\left|\sum_n a_n\bar{c}_n\right|\leq M\norm{\sum_n a_nf_n}{H}
\end{align}
for every finite sequence of scalars $(a_n)_n$.
\end{thm} 

\noindent We are now ready to give the proof of the main controllability result.

\begin{proof}[\bf Proof of Theorem \ref{control_thm}]
Using Lemma \ref{moment_lemma}, it is sufficient to show that, for each initial data $(y^0,y^1)\in\mathbb H_s^{3}(\I)\times\mathbb H_s^{2}(\I)$ given by \eqref{data_dec}, there exists $\widetilde{u}\in L^2(\Q)$ such that \eqref{moment_pb} is satisfied. From Theorem \ref{thm_young}, this is equivalent to show that

\begin{align}\label{young_ineq}
	\left| \sum_{(n,j)\in S} (\mu_{|n|}^j\bar{y}_n^0+\bar{y}_n^1)a_n^j\right|^2 \leq C\int_0^T\int_{\omega_0} \left|\sum_{(n,j)\in S} a_n^je^{i\,\sgn(n)c\rho_{|n|}^{\frac{1}{2s}}x}e^{-\lambda_n^jt}\right|^2\,dxdt,
\end{align} 
for all finite sequence $(a_n^j)_{(n,j)\in S}$. We notice that
\begin{align}\label{ineq1}
	\left| \sum_{(n,j)\in S} (\mu_{|n|}^j\bar{y}_n^0+\bar{y}_n^1)a_n^j\right|^2 &\leq \left(\sum_{(n,j)\in S} \rho_{|n|}^2\left|\mu_{|n|}^j\bar{y}_n^0+\bar{y}_n^1\right|^2\right)\left(\sum_{(n,j)\in S} \frac{|a_n^j|^2}{\rho_{|n|}^2}\right) \notag 
	\\
	&\leq C\norm{(y^0,y^1)}{\mathbb H_s^{3}(\I)\times\mathbb H_s^{2}(\I)}^2 \left(\sum_{(n,j)\in S} \frac{|a_n^j|^2}{\rho_{|n|}^2}\right).
\end{align}

Let us mention that, if $c\notin{\mathcal V}$, then all the eigenvalues are simple and the separation of the second term in the last inequality is not needed. 
On the other hand,  using \eqref{eq:insum} and taking into account that we can have at most one double eigenvalue $\lambda^2_{-n_c}$ (if $c\in {\mathcal V}$, see Lemma \ref{lemma:dist20}), we can deduce that
\begin{align*}
	\int_0^T&\int_{\omega_0}\left| \sum_{(n,j)\in S} a_n^j e^{i\,\sgn(n)c\rho_{|n|}^{\frac{1}{2s}}x}  e^{-\lambda_n^jt}\right|^2\,dxdt = \int_{\omega_0}\int_{-\frac{T}{2}}^{\frac{T}{2}}\left| \sum_{(n,j)\in S} a_n^j e^{\lambda_n^j \frac{T}{2}} e^{i\,\sgn(n)c\rho_{|n|}^{\frac{1}{2s}}x}  e^{-\lambda_n^j t}\right|^2\,dtdx
	\\
	&\geq C \left(|\omega_0|
	\sum_{(n,j)\in S\setminus\{(-n_c,2),\,(n_c,3)\}}\frac{|a_n^j|^2}{\rho_{|n|}^2} + \int_{\omega_0} \left|a_{-n_c}^2 e^{\lambda_{-n_c}^2 \frac{T}{2}} e^{-i\,\rho_{|n_c|}^{\frac{1}{2s}}x} +  a_{n_c}^3 e^{\lambda_{n_c}^3 \frac{T}{2}} e^{i\,\rho_{|n_c|}^{\frac{1}{2s}}x}\right|^2 \,dx\right)\!.
\end{align*}
Since the map
\begin{align*}
	\CC^2\ni(a,b)\mapsto \left(\int_{\omega_0}\left|ae^{-i\,\rho_{|n_c|}^{\frac{1}{2s}}x} + be^{i\,\rho_{|n_c|}^{\frac{1}{2s}}x}\right|^2\,dx\right)^{\frac 12}
\end{align*}
is a norm in $\CC^2$, it follows that 
\begin{align}\label{ineq2}
	\int_0^T\int_{\omega_0}\left| \sum_{(n,j)\in S} a_n^j e^{i\,\sgn(n)\rho_{|n|}^{\frac{1}{2s}}x}  e^{-\lambda_n^jt}\right|^2\,dxdt \geq C \sum_{(n,j)\in S}\frac{|a_n^j|^2}{\rho_{|n|}^2}.
\end{align}
From \eqref{ineq1} and \eqref{ineq2} we immediately obtain \eqref{young_ineq}. The proof is finished.
\end{proof}

We conclude the paper with the following observation.

\begin{rem} 
{\em
We notice the following facts.
\begin{enumerate}
	\item[(a)] We have proved that the control time has to be larger than
	\begin{align*}
		T=2\pi\left(\frac{1}{|c|}+\frac{1}{|c+\gamma|}+\frac{1}{|c-\gamma|}\right).
	\end{align*}
This result is a consequence of the particular construction on the biorthogonal sequence in Theorem \ref{te:bio}. The determination of the optimal control time remains an interesting open problem.
	\item[(b)] Let us mention that the space of controllable initial data is larger than the one given by Theorem \ref{control_thm}. This is a consequence of the fact that  the small weight $\frac{1}{\rho_{|n|}^2}$ in  \eqref{ineq2} affects only some terms in the right hand side series. However, it is not easy to identify a larger classical space of controllable initial data than $\mathbb H_s^{3}(\I)\times\mathbb H_s^{2}(\I)$.
	\item[(c)] We already anticipated that, as it is expectable, when $M=0$ the spectrum of our operator coincides with the one of the fractional wave operator without memory. For this equation, it was proved in \cite{biccari2018internal} that null controllability is not achievable for any $s\in(0,1)$ (see also \cite{warma2018analysis}, where analogous results were obtained for a strongly damped fractional wave equation). Nevertheless, in the mentioned papers the authors always consider controls supported in fixed regions, either in the interior or in the exterior of the domain. On the other hand, the results we obtained in this work show that controllability may be achieved for any $\frac 12<s<1$, provided that the control is moving in time.
\end{enumerate}
}
\end{rem}

\noindent{\bf Acknowledgments}:
\begin{itemize}
\item The work of the first author is supported by the European Research Council (ERC) under the European Union’s Horizon 2020 research and innovation programme (grant agreement NO: 694126-DyCon). 
\item The work of the first author was partially supported by the Grants MTM2014-52347, MTM2017-92996 and MTM2017-82996-C2-1-R COSNET of MINECO (Spain), and by the ELKARTEK project KK-2018/00083 ROAD2DC of the Basque Government.
\item The work of both authors is supported by the Air Force Office of Scientific Research (AFOSR) under Award NO:  FA9550-18-1-0242. 
\item Part of this work was done during the first author's visit at the University of Puerto, Rico Rio Piedras Campus, and he would like to thank the members of this institution for their kindness and warm hospitality. 
\end{itemize}

\bibliographystyle{plain}
\bibliography{biblio}

\end{document}